\documentclass[10pt]{article}
\usepackage{amsmath}
\usepackage{amsfonts}
\usepackage{stmaryrd}
\usepackage{mathrsfs}
\usepackage{amssymb,amsfonts,amsmath, psfrag,eepic,colordvi,epsfig}
\usepackage{enumerate}
\usepackage{cite}
\usepackage{longtable}
\usepackage{hyperref}
\topskip 
\numberwithin{equation}{section}
\newtheorem{theorem}{Theorem}[section]

\newtheorem{remark}[theorem]{Remark}

\begin{document}
	\parskip 7pt
	
	\pagenumbering{arabic}
	\def\sof{\hfill\rule{2mm}{2mm}}
	\def\ls{\leq}
	\def\gs{\geq}
	\def\SS{\mathcal S}
	\def\qq{{\bold q}}
	\def\MM{\mathcal M}
	\def\TT{\mathcal T}
	\def\EE{\mathcal E}
	\def\lsp{\mbox{lsp}}
	\def\rsp{\mbox{rsp}}
	\def\pf{\noindent {\it Proof.} }
	\def\mp{\mbox{pyramid}}
	\def\mb{\mbox{block}}
	\def\mc{\mbox{cross}}
	\def\qed{\hfill \rule{4pt}{7pt}}
	\def\pf{\noindent {\it Proof.} }
	\textheight=21cm
	
	\begin{center}
		{\Large\bf   Proofs of some conjectures by Bringmann, Han, Heim, and Kane on periodic sign changes in 
		 eta-quotients }
	\end{center}
	
	\begin{center}
	
	 Jing Jin$^{1}$,       Olivia X.M. Yao$^{2}$
	  and Shiwen Zhou$^{3}$
	\\[6pt]

	$^{1}$College of Agricultural Information,\\
	Jiangsu Agri-animal
	Husbandry Vocational College,
	\\
	Taizhou, 225300,  Jiangsu,  P. R. China
	
	$^{2,3}$School of Mathematical Sciences, \\
	Suzhou University of Science and
	Technology, \\
	Suzhou,  215009, Jiangsu,
	P. R. China

	Email:  jinjing19841@126.com,    
	  yaoxiangmei@163.com,  2870279931@qq.com

\end{center}

	\noindent {\bf Abstract.}
	 Bringmann, Han, Heim, and Kane recently investigated periodic sign patterns  in the Fourier coefficients of
	weakly holomorphic modular forms arising from eta-quotients  and posed numerous conjectures in the
	appendix to their work. Using classical theta-function identities, dissections, and coefficientwise positivity,
	we prove 72 of these conjectural sign patterns. The eta-quotients considered here have periods 2, 3, 4, 5,
	6, 8, and 12. Our proofs are elementary at the level of $q$-series and yield explicit generating functions for
	the coefficients in each relevant residue class.


	
	\section{Introduction}
	\allowdisplaybreaks
	
	Throughout this paper, 
	 define 
	\begin{align}
		\sum_{n=0}^\infty 
		a_{1^{t_1}2^{t_2}\cdots m^{t_m}}(n)q^n=f_1^{t_1}f_2^{t_2}\cdots
		 f_m^{t_m}, \label{1-1}
	\end{align}
	where $m$ is a positive integer, 
	  $t_1,t_2,\ldots, t_m$ are integers and $f_j:=(q^j;q^j)_\infty$.
 Here and throughout, we assume that  $q$
 is a complex number with $|q|<1$ 
 and adopt the following standard
 $q$-series notation:
 \begin{align*}
 	(a;q)_\infty:=&\prod_{k=0}^\infty
 	(1-aq^k),\qquad
 	(a_1,a_2,\ldots ,a_k;q)_\infty :=(a_1;q)_\infty
 	 (a_2;q)_\infty \cdots (a_k;q)_\infty.
 \end{align*}

	 Many interesting examples of sign changes of the $	a_{1^{t_1}2^{t_2}\cdots m^{t_m}}(n)$ have 
   appeared  in the literature over the past thirty years. For example, 
   Andrews \cite{Andrews} studied the sign pattern of $a_{1^{1}p^{-1}}(n)$ and proved  that for all
   primes $p$, $a_{1^{1}p^{-1}}(n)$ and $a_{1^{1}p^{-1}}(n+p)$ have the same sign for each $n\geq 0$.  Schlosser and Zhou \cite{Schlosser} showed  that 
    for all primes $p$, $a_{1^{3}p^{-3}}(n)$ and $a_{1^{3}p^{-3}}(n+p)$ have the same sign for each $n\geq 0$.  For  positive integers $t$
   		and $m$    with $m(t-1)\leq 24$,  Wang \cite{Wang} gave an asymptotic formula for
   	$a_{1^{m}t^{-m}}(n)$, and presented characterizations of $n$ for which  	$a_{1^{m}t^{-m}}(n)$ is positive, negative or zero.
   	 In 2000, Andrews and Lewis \cite{Andrews-Lewis} conjectured that 
   	 \begin{align*}
   	 		  	a_{1^23^{-1}}(n) 
   	 	\left\{\begin{array}{lll}
   	 		  >0,
   	 		& \text {\rm  if }
   	 		3|n\ {\rm or}\ n=5, \\
   	 		  =0,
   	 		& \text {\rm  if }
   	 		 \ n\in\{14,17\}, \\
   	 		<0, & \text {\rm otherwise.}
   	 	\end{array}\right.  
   	 \end{align*}
This conjecture was confirmed by Kane \cite{Kane}. 

 Very recently,
 Bringmann, Han, Heim and Kane \cite{Bringmann-1} studied the signs of $a_{1^{t_1}2^{t_2}\cdots m^{t_m}}(n)$
  and 
  gave representative examples for forms of negative weight, weight
 zero, and positive weight. For example, they showed that 
  for $n\geq 0$, 
     	 \begin{align*}
  	a_{1^93^{-5}}(n) 
  	\left\{\begin{array}{lll}
  		>0,
  		& \text {\rm  if }\ 
  		 n\equiv 0,2,5,6,8\pmod 9, \\
  		<0,
  	& \text {\rm  if }\ 
  	n\equiv1,3,4,7\pmod 9. 
  	\end{array}\right.  
  \end{align*}
  
  In a separate paper, Bringmann, Han, Heim, and Kane  \cite{Bringmann-2}  investigated sign changes in an infinite family of
  holomorphic eta-quotients. In Appendix A of   \cite{Bringmann-1},  the same authors listed numerous conjectural period sign
  patterns for weakly holomorphic eta-quotients. In the present paper, we establish 72 of those patterns by
  combining classical theta-function identities with coefficientwise positivity arguments. The results are
  organized by period: Section 2 treats period 2, Section 3 period 3, Section 4 period 4, Section 5 period 5,
  Section 6 period 6, Section 7 period 8, and Section 8 period 12.
  
  For reference, Table~\ref{tab:proved-cases} records precisely the 72 cases from Appendix A of
  \cite{Bringmann-1} that are proved below. In the sign-pattern column, the symbols are ordered by
  residue class modulo the displayed period; factors with exponent zero are omitted from the case notation.
  Each row corresponds to one theorem, and the final column gives the number of cases covered by that theorem.
  
  \begingroup
  \footnotesize
  \setlength{\tabcolsep}{3pt}
  \renewcommand{\arraystretch}{1.08}
  \begin{longtable}{|c|c|c|p{7.2cm}|c|}
  	\caption{The 72 proved cases, grouped by theorem.}
  	\label{tab:proved-cases}\\
  	\hline
  	Theorem & Period & Sign pattern & \multicolumn{1}{c|}{Cases} & Count\\
  	\hline
  	\endfirsthead
  	\multicolumn{5}{c}{\tablename~\thetable\ (continued)}\\
  	\hline
  	Theorem & Period & Sign pattern & \multicolumn{1}{c|}{Cases} & Count\\
  	\hline
  	\endhead
  	\hline
  	\multicolumn{5}{r}{\footnotesize Continued on next page}\\
  	\endfoot
  	\hline
  	\endlastfoot
  	Theorem~\ref{Th-1} & 2 & \texttt{+-} &
  	\(1^{2}2^{-3}3^{-3}\), \(1^{2}2^{-2}3^{4}\), \(1^{3}2^{-2}\),
  	\(1^{3}2^{-2}3^{1}\), \(1^{3}2^{-2}3^{2}\), \(1^{3}2^{-2}3^{3}\),
  	\(1^{4}2^{-4}3^{-4}\), \(1^{4}2^{-3}3^{4}\) & 8\\ \hline
  	Theorem~\ref{Th-2} & 3 & \texttt{+0-} &
  	\(2^{2}3^{-1}4^{-1}\), \(2^{3}3^{-1}\) & 2\\ \hline
  	Theorem~\ref{Th-3} & 3 & \texttt{+-0} &
  	\(1^{3}3^{-2}\), \(1^{1}2^{-1}3^{-2}4^{1}\), \(1^{2}2^{-1}3^{-2}\) & 3\\ \hline
  	Theorem~\ref{Th-4} & 3 & \texttt{+--} &
  	\(1^{1}2^{1}3^{-2}\) & 1\\ \hline
  	Theorem~\ref{Th-4-1} & 4 & \texttt{+++-} &
  	\(1^{-1}2^{1}3^{3}4^{-3}\), \(1^{-3}2^{8}3^{1}4^{-7}\) & 2\\ \hline
  	Theorem~\ref{Th-4-2} & 4 & \texttt{++--} &
  	\(1^{-1}2^{3}3^{1}4^{-3}\), \(1^{-1}2^{4}3^{1}4^{-3}\),
  	\(1^{-1}2^{4}4^{-4}\) & 3\\ \hline
  	Theorem~\ref{Th-4-3} & 4 & \texttt{++0+} &
  	\(1^{-1}2^{2}4^{-2}5^{1}\) & 1\\ \hline
  	Theorem~\ref{Th-4-4} & 4 & \texttt{+--+} &
  	\(1^{1}3^{-1}4^{-3}\), \(1^{1}2^{1}3^{-1}4^{-3}\),
  	\(1^{1}2^{2}3^{-1}4^{-3}\), \(1^{1}2^{3}3^{-1}4^{-3}\),
  	\(1^{1}2^{4}3^{-1}4^{-4}\), \(1^{2}4^{-2}\),
  	\(1^{2}2^{1}4^{-2}\), \(1^{2}2^{2}4^{-2}\),
  	\(1^{2}2^{3}4^{-3}\), \(1^{2}2^{4}4^{-3}\),
  	\(1^{1}2^{1}4^{-3}\), \(1^{1}2^{2}4^{-3}\),
  	\(1^{1}2^{3}4^{-3}\), \(1^{1}2^{4}4^{-3}\) & 14\\ \hline
  	Theorem~\ref{Th-4-5} & 4 & \texttt{+---} &
  	\(1^{1}3^{1}4^{-3}\) & 1\\ \hline
  	Theorem~\ref{Th-4-6} & 4 & \texttt{+-00} &
  	\(1^{2}2^{-1}4^{-1}\) & 1\\ \hline
  	Theorem~\ref{Th-4-7} & 4 & \texttt{+--0} &
  	\(1^{2}3^{2}4^{-3}\) & 1\\ \hline
  	Theorem~\ref{Th-4-8} & 4 & \texttt{+-++} &
  	\(1^{3}2^{-1}3^{-1}4^{-2}\), \(1^{4}2^{-1}4^{-3}\),
  	\(1^{4}4^{-6}\), \(1^{4}2^{1}4^{-9}\) & 4\\ \hline
  	Theorem~\ref{Th-4-9} & 4 & \texttt{+-0+} &
  	\(1^{3}3^{-1}4^{-4}\) & 1\\ \hline
  	Theorem~\ref{Th-4-10} & 4 & \texttt{+-+0} &
  	\(1^{4}2^{-2}4^{-1}\) & 1\\ \hline
  	Theorem~\ref{Th-5-1} & 5 & \texttt{++0+0} &
  	\(1^{-1}2^{2}5^{-1}\) & 1\\ \hline
  	Theorem~\ref{Th-5-2} & 5 & \texttt{++-00} &
  	\(1^{-1}2^{3}4^{-1}5^{-2}\) & 1\\ \hline
  	Theorem~\ref{Th-5-3} & 5 & \texttt{+0-0-} &
  	\(2^{1}5^{-1}\) & 1\\ \hline
  	Theorem~\ref{Th-5-4} & 5 & \texttt{+-00+} &
  	\(1^{2}2^{-1}5^{-2}\) & 1\\ \hline
  	Theorem~\ref{Th-6-1} & 6 & \texttt{+-+-+0} &
  	\(1^{1}2^{-4}3^{-3}4^{4}\), \(1^{3}2^{-3}3^{-1}4^{3}\) & 2\\ \hline
  	Theorem~\ref{Th-6-2} & 6 & \texttt{+-+0+0} &
  	\(1^{1}2^{-3}3^{-3}4^{2}\) & 1\\ \hline
  	Theorem~\ref{Th-6-3} & 6 & \texttt{+-++0-} &
  	\(1^{1}2^{-2}3^{-3}\) & 1\\ \hline
  	Theorem~\ref{Th-7-1} & 8 & \texttt{+-++0-0+} &
  	\(1^{4}4^{-5}\) & 1\\ \hline
  	Theorem~\ref{Th-7-2} & 8 & \texttt{+-++---+} &
  	\(1^{4}4^{-4}\), \(1^{4}4^{-3}\), \(1^{4}4^{-2}\),
  	\(1^{4}2^{1}4^{-6}\), \(1^{4}2^{1}4^{-5}\),
  	\(1^{4}2^{1}4^{-4}\), \(1^{4}2^{1}4^{-3}\) & 7\\ \hline
  	Theorem~\ref{Th-7-3} & 8 & \texttt{+-0+0-0+} &
  	\(1^{4}2^{2}4^{-10}\) & 1\\ \hline
  	Theorem~\ref{Th-7-4-0} & 8 & \texttt{+-0+--0+} &
  	\(1^{4}2^{2}4^{-9}\), \(1^{4}2^{2}4^{-8}\),
  	\(1^{4}2^{2}4^{-7}\), \(1^{4}2^{2}4^{-6}\),
  	\(1^{4}2^{2}4^{-5}\), \(1^{4}2^{2}4^{-4}\),
  	\(1^{4}2^{2}4^{-3}\), \(1^{4}2^{2}4^{-2}\) & 8\\ \hline
  	Theorem~\ref{Th-7-4} & 8 & \texttt{+--+--++} &
  	\(1^{4}2^{3}4^{-4}\), \(1^{4}2^{4}4^{-4}\),
  	\(1^{4}2^{4}4^{-3}\) & 3\\ \hline
  	Theorem~\ref{Th-8-1} & 12 & \texttt{++0-0+--0+0-} &
  	\(1^{-3}2^{9}3^{1}4^{-6}\) & 1\\
  \end{longtable}
  \endgroup

	\section{Sign changes with period 2}
	
	In this section, we prove several results 
	 on sign changes with period 2 which were conjectured
	  by  Bringmann, Han, Heim and Kane
	  \cite{Bringmann-1}.
	  
     	  Let $f(q)=\sum_{n=0}^\infty u(n) q^n$
 and $g(q)=\sum_{n=0}^\infty v(n) q^n$  be two formal power series in $q$. Throughout this paper,  we write $f(q)\succcurlyeq g(q)$ (resp. $f(q) \preccurlyeq  g(q)$)
 if for all integers $n\geq 0$, $u(n)\geq v(n)$ (resp. $u(n) \leq 
 v(n)$). We require  the following fact.  Let $M$ be a positive  integer and let $F(q)$ be a  formal power series with nonnegative coefficients.
  If $F(q)\succcurlyeq 1+q+\cdots +q^{M-1}$,
   then 
   \[
   \frac{F(q)}{1-q^M} \succcurlyeq \sum_{n=0}^\infty q^n, \qquad 
     - \frac{F(q)}{1-q^M} \preccurlyeq - \sum_{n=0}^\infty q^n.
   \]
	
	\begin{theorem}\label{Th-1}
		Define
		\[S_1:=\{(2,-3,-3),(2,-2,4),(3,-2,0),(3,-2,1),(3,-2,2), (3,-2,3),(4,-4,-4),(4,-3,4)\}.\]
		If $(t_1,t_2,t_3)\in S_1$, then for $n\geq 0$,
		\begin{align*}
		a_{1^{t_1}2^{t_2}3^{t_3}}(2n)&>0, \quad 
			a_{1^{t_1}2^{t_2}3^{t_3}}(2n+1)<0. 
		\end{align*}
	\end{theorem}

\noindent{\it Proof.}
	In view of \eqref{1-1}, 
 \begin{align}
 \sum_{n=0}^\infty a_{1^22^{-3}3^{-3}}(n)q^n =\frac{f_1^2}{f_2^3f_3^3}.\label{2-3}
 \end{align}
 Replacing $q$ by $-q$ in \eqref{2-3} and using 
  the following identity 
 \[
 (-q;-q)_\infty =\frac{f_2^3}{f_1f_4},
 \]
 we arrive at 
\begin{align}
 \sum_{n=0}^\infty (-1)^n a_{1^22^{-3}3^{-3}}(n)q^n =
 \frac{f_2f_3^2}{f_1f_6}\cdot
 \frac{f_2^2f_3f_{12}}{f_1f_4f_6} \cdot \frac{1}{f_4}
 \cdot \frac{1}{f_6^5}\cdot \frac{1}{(q^6;q^{12})_\infty^2}.\label{2-4}
 \end{align}
By the Jacobi triple product identity,  
\begin{align}\label{2-5}
\sum_{n=-\infty}^\infty q^{3n^2+2n}=\frac{f_2^2f_3f_{12}}{f_1f_4f_6} 
 \end{align}
and 
\begin{align}\label{2-6}
\sum_{n=-\infty}^\infty q^{n(3n+1)/2}=\frac{f_2f_3^2}{f_1f_6} .
 \end{align}
 In view of \eqref{2-4}--\eqref{2-6}, 
 \begin{align*}
 	\sum_{n=0}^\infty (-1)^n a_{1^22^{-3}3^{-3}}(n)q^n& =
 \left(\sum_{n=-\infty}^\infty q^{n(3n+1)/2}\right) \cdot
 	\left(\sum_{n=-\infty}^\infty q^{3n^2+2n}\right) \cdot \frac{1}{f_4}
 	\cdot \frac{1}{f_6^5}\cdot \frac{1}{(q^6;q^{12})_\infty^2}
 	\nonumber\\
 	&\succcurlyeq (1+q+q^2)(1+q) \cdot \frac{1}{1-q^4}
 	\succcurlyeq \sum_{n=0}^\infty q^n , 
 \end{align*}
 which implies 
  that Theorem \ref{Th-1} is  true 
   when $(t_1,t_2,t_3 )=(2,-3,-3)$.

By \eqref{1-1}, 
 \begin{align}
\sum_{n=0}^\infty a_{1^22^{-2}3^4}(n)q^n
  =\frac{f_1^2f_3^4}{f_2^2}.\label{2-8}
 \end{align}
 Replacing $q$ by $-q$ in \eqref{2-8}
  yields 
  \begin{align}
\sum_{n=0}^\infty(-1)^n  a_{1^22^{-2}3^4}(n)q^n
=\frac{f_2^5}{f_1^2f_4^2} \cdot \frac{f_6^{10}}{f_3^4f_{12}^4}
\cdot \frac{f_6^3}{f_2} \cdot \frac{1}{f_6}.\label{2-9}
\end{align}
Gauss proved that 
\begin{align}\label{2-10}
	\sum_{n=-\infty}^\infty q^{n^2}=\frac{f_2^5}{f_1^2f_4^2} 
\end{align}
and 
\begin{align}\label{2-10-1}
	\sum_{n=0}^\infty q^{n(n+1)/2}=\frac{f_2^2}{f_1 } .
\end{align}
Moreover, let $c_t(n)$ denote the number of $t$-core partitions
 of $n$. The generating function
  of $c_t(n) $ is
  \begin{align}
  	\sum_{n=0}^\infty c_t(n)q^n =\frac{f_t^t}{f_1}.\label{2-11}
  \end{align}

  Combining \eqref{2-9}, \eqref{2-10} and \eqref{2-11} yields
    \begin{align*}
  	\sum_{n=0}^\infty(-1)^n  a_{1^22^{-2}3^4}(n)q^n
  	&=\left(\sum_{n=-\infty}^\infty q^{n^2}\right)\cdot \left(\sum_{n=-\infty}^\infty q^{3n^2}\right)^2
  	\cdot \left(\sum_{n=0}^\infty c_3(n)q^{2n} \right)\cdot \frac{1}{f_6}\nonumber\\
  	&\succcurlyeq (1+2q+2q^4)(1+2q^3)^2(1+q^2)\cdot
  	\frac{1}{1-q^6} \succcurlyeq \sum_{n=0}^\infty q^n,  
  \end{align*}
   which implies 
  that Theorem \ref{Th-1} is true 
  when $(t_1,t_2,t_3 )=(2,-2,4)$.

Thanks to \eqref{1-1}, 
    \begin{align}\label{2-13}
\sum_{n=0}^\infty a_{1^32^{-2}}(n)q^n
=\frac{f_1^3}{f_2^2}.
\end{align}
Replacing $q$ by $-q$ in \eqref{2-13}
 and using \eqref{2-10} and \eqref{2-10-1} yields 
    \begin{align*}
\sum_{n=0}^\infty (-1)^n a_{1^32^{-2}}(n)q^n
&=\frac{f_2^5}{f_1^2f_4^2}\cdot \frac{f_2^2}{f_1}\cdot \frac{1}{f_4}
\nonumber\\
&=\left(\sum_{n=-\infty}^\infty q^{n^2}\right)\cdot \left(\sum_{n=0}^\infty q^{n(n+1)/2}\right)\cdot \frac{1}{f_4}
\nonumber\\
&\succcurlyeq (1+2q+2q^4)(1+q+q^3)\cdot \frac{1}{1-q^4}\succcurlyeq 
\sum_{n=0}^\infty q^n,
\end{align*}
 which implies 
that Theorem \ref{Th-1}  holds 
when $(t_1,t_2,t_3 )=(3,-2,0)$.

In light of \eqref{1-1},
\begin{align}\label{2-16}
\sum_{n=0}^\infty a_{1^32^{-2}3^1}(n)q^n
=\frac{f_1^3f_3}{f_2^2}.
\end{align}
Replacing $q$ by $-q$ in \eqref{2-16}
 and utilizing \eqref{2-5}, \eqref{2-10} and  \eqref{2-10-1}, we obtain 
\begin{align*} 
\sum_{n=0}^\infty (-1)^n a_{1^32^{-2}3^1}(n)q^n
&= \frac{f_2^2f_3f_{12}}{f_1f_4f_6}\cdot \frac{f_2^5}{f_1^2f_4^2} 
\cdot \frac{f_6^4}{f_3^2}\cdot \frac{1}{f_{12}^2}
\nonumber\\
&=\left(\sum_{n=-\infty}^\infty q^{3n^2+2n}\right)
\cdot \left(\sum_{n=-\infty}^\infty q^{n^2}\right)
\left(\sum_{n=0}^\infty q^{3n(n+1)/2}\right)^2\cdot \frac{1}{f_{12}^2}
\nonumber\\
&\succcurlyeq (1+q+q^5+q^8)(1+2q+2q^4+2q^9) (1+q^3)^2\cdot
\frac{1}{1-q^{12}} \succcurlyeq \sum_{n=0}^\infty q^n,
\end{align*}
 which implies 
that Theorem \ref{Th-1} is true 
when $(t_1,t_2,t_3 )=(3,-2,1)$.

In view of \eqref{1-1}, 
\begin{align}\label{2-18}
\sum_{n=0}^\infty a_{1^32^{-2}3^2}(n)q^n
=\frac{f_1^3f_3^2}{f_2^2}.
\end{align}
Replacing $q$ by $-q$ in \eqref{2-18} 
 and using \eqref{2-5}, \eqref{2-10}
  and \eqref{2-10-1} yields 
\begin{align*} 
\sum_{n=0}^\infty (-1)^n a_{1^32^{-2}3^2}(n)q^n
&=\frac{f_2^5}{f_1^2f_4^2} \cdot \frac{f_6^5}{
f_3^2f_{12}^2}\cdot \frac{f_2^2f_3f_{12}}{f_1f_4f_6}\cdot \frac{f_6^2}{f_3}\cdot\frac{1}{f_{12}}\nonumber\\
&=\left(\sum_{n=-\infty}^\infty q^{n^2}\right) \cdot\left(\sum_{n=-\infty}^\infty q^{3n^2}\right)\cdot \left(\sum_{n=-\infty}^\infty q^{3n^2+2n}\right)\cdot \left(\sum_{n=0}^\infty q^{3n(n+1)/2}\right)\cdot\frac{1}{f_{12}}
\nonumber\\
&\succcurlyeq \left(1+2q+2q^4+2q^9 \right) \cdot\left(1+2q^3 \right)\cdot \left(1+q+q^5 +q^8 \right)\cdot \left(1+q^3 \right)\cdot\frac{1}{1-q^{12}}
\nonumber\\
&\succcurlyeq \sum_{n=0}^\infty q^n ,
\end{align*}
 which implies 
that Theorem \ref{Th-1} is true 
when $(t_1,t_2,t_3 )=(3,-2,2)$.

Thanks to \eqref{1-1}, 
\begin{align}\label{2-20}
\sum_{n=0}^\infty a_{1^32^{-2}3^3}(n)q^n
=\frac{f_1^3f_3^3}{f_2^2}.
\end{align}
Replacing $q$ by $-q$ in \eqref{2-20}
 and using \eqref{2-5} and \eqref{2-10}, we arrive at 
\begin{align}\label{2-21}
\sum_{n=0}^\infty (-1)^n a_{1^32^{-2}3^3}(n)q^n
&=\frac{f_2^5}{f_1^2f_4^2}
\cdot \frac{f_6^{10}}{f_3^4f_{12}^4} \cdot \frac{f_2^2f_3f_{12}}{f_1f_4f_6}
\nonumber\\
&=
\left(\sum_{n=-\infty}^\infty q^{n^2}\right) \cdot\left(\sum_{n=-\infty}^\infty q^{3n^2}\right)^2
 \cdot\left(\sum_{n=-\infty}^\infty q^{3n^2+2n}\right)
\end{align}
In \cite{Ju}, Ju and Oh proved that 
\begin{align}\label{2-22}
	\left(\sum_{n=-\infty}^\infty q^{n^2}\right) \cdot\left(\sum_{n=-\infty}^\infty q^{3n^2}\right)^2
	\cdot\left(\sum_{n=-\infty}^\infty q^{3n^2+2n}\right)\succcurlyeq
	\sum_{n=0}^\infty q^n. 
\end{align}
It follows from \eqref{2-21}
 and \eqref{2-22} that 
\begin{align*} 
	\sum_{n=0}^\infty (-1)^n a_{1^32^{-2}3^3}(n)q^n
	\succcurlyeq
	\sum_{n=0}^\infty q^n. 
\end{align*}
Thus,  Theorem \ref{Th-1} is true 
when $(t_1,t_2,t_3 )=(3,-2,3)$.

In view of \eqref{1-1},
\begin{align}\label{2-24}
\sum_{n=0}^\infty a_{1^42^{-4}3^{-4}}(n)q^n
=\frac{f_1^4 }{f_2^4f_3^4}.
\end{align}
Replacing $q$ by $-q$
 in \eqref{2-24} and using \eqref{2-5}, we have 
\begin{align*} 
\sum_{n=0}^\infty (-1)^n a_{1^42^{-4}3^{-4}}(n)q^n
&=\left(\frac{f_2^2f_3f_{12}}{f_1f_4f_6}\right)^4\frac{1}{f_6^8}
\nonumber\\
&=\left(\sum_{n=-\infty }^\infty q^{3n^2+2n}\right)^4\frac{1}{f_6^8}
\nonumber\\
&\succcurlyeq (1+q+q^5)^4\cdot \frac{1}{1-q^6} \succcurlyeq \sum_{n=0}^\infty q^n,
\end{align*}
 which implies 
that Theorem \ref{Th-1} is true 
when $(t_1,t_2,t_3 )=(4,-4,-4)$.

In light of \eqref{1-1},
\begin{align}\label{2-26}
\sum_{n=0}^\infty a_{1^42^{-3}3^{4}}(n)q^n
=\frac{f_1^4 f_3^4}{f_2^3 }.
\end{align}
Replacing $q$ by $-q$  in \eqref{2-26} and using \eqref{2-10} and \eqref{2-11},
 we get 
\begin{align*} 
\sum_{n=0}^\infty (-1)^n a_{1^42^{-3}3^{4}}(n)q^n
&=\frac{f_2^{10}}{f_1^4f_4^4}
\frac{f_6^{10}}{f_3^4f_{12}^4} \frac{f_6^3}{f_2} \frac{1}{f_6}
\nonumber\\
&=\left(\sum_{n=-\infty}^\infty q^{n^2}\right)^2 \cdot  \left(\sum_{n=-\infty}^\infty q^{3n^2}\right)^2
\cdot \left(\sum_{n=0}^\infty c_3(n)q^{2n}\right) \cdot  \frac{1}{f_6}
\nonumber\\
&\succcurlyeq (1+2q+2q^4)^2 (1+2q^3)^2(1+q^2)\cdot \frac{1}{1-q^6} \succcurlyeq
\sum_{n=0}^\infty q^n,
\end{align*}
 which implies 
that Theorem \ref{Th-1} is  true 
when $(t_1,t_2,t_3 )=(4,-3,4)$.  This completes
 the proof of Theorem \ref{Th-1}. \qed

	\section{Sign changes with period 3}
	
	The aim of this section is to confirm several  conjectures 
 of  Bringmann, Han, Heim and Kane
 \cite{Bringmann-1}	on sign changes with period 3. 

\begin{theorem}\label{Th-2}
	Define
	\[
	S_2:=\{(2,-1,-1),(3,-1,0)\}.
	\]
If $(t_2,t_3,t_4)\in S_2$, then for $n\geq 0$,
	\begin{align*}
		a_{2^{t_2}3^{t_3}4^{t_4}}(3n) >0,\quad 
		a_{2^{t_2}3^{t_3}4^{t_4}}(3n+1) =0,\quad a_{2^{t_2}3^{t_3}4^{t_4}}(3n+2) <0.
	\end{align*}
\end{theorem}

\noindent{\it Proof}. By \eqref{1-1}, 
	\begin{align}\label{3-2}
\sum_{n=0}^\infty a_{2^23^{-1}4^{-1}}(n)q^n=\frac{f_2^2}{f_3f_4}.
	\end{align}
It follows from \cite[Corollary (i), p. 49]{Berndt} that 
	\begin{align}\label{3-3}
\frac{f_1^2}{f_2}=\frac{f_9^2}{f_{18}}-2q\frac{f_3f_{18 }^2}{f_6f_9}.
\end{align}
Combining \eqref{3-2} and \eqref{3-3}
 yields 
\[ 	
 	\sum_{n=0}^\infty a_{2^23^{-1}4^{-1}}(n)q^n=\frac{f_{18}^2}{f_3f_{36}}-2q^2\frac{f_6
 		 f_{36 }^2}{f_3f_{12}f_{18}},
\]
from which, we arrive at 
\begin{align}
\sum_{n=0}^\infty a_{2^23^{-1}4^{-1}}(3n)q^n&=
\frac{f_6^2}{f_1f_{12}} =\frac{1}{(q,q^2;q^3)_\infty f_{12}}
\cdot \left(\sum_{n=0}^\infty q^{3n(n+1)/2}\right)
 \quad {\rm (by \ \eqref{2-10-1})}
\nonumber\\
&\succcurlyeq \sum_{n=0}^\infty q^n,   \label{3-4}
\\
\sum_{n=0}^\infty a_{2^23^{-1}4^{-1}}(3n+1)q^n&=
0
, \label{3-5}\\
\sum_{n=0}^\infty a_{2^23^{-1}4^{-1}}(3n+2)q^n&=
-2\frac{f_2f_{12}^2}{f_1f_4f_6} =-2\frac{1}{(q;q^2)_\infty
 f_4}\cdot  \left(\sum_{n=0}^\infty q^{3n(n+1)}\right)
 \quad {\rm (by \ \eqref{2-10-1})}
 \nonumber\\
 &\preccurlyeq - \sum_{n=0}^\infty q^n.   \label{3-6}
\end{align}
It follows from \eqref{3-4}--\eqref{3-6}
 that Theorem \ref{Th-2}
  is true when $(t_2,t_3,t_4)=(2,-1,-1)$.

In view of \eqref{1-1}, 
\begin{align}
\sum_{n=0}^\infty a_{2^33^{-1} }(n)q^n=\frac{f_2^3}{f_3}.
\label{3-7}
\end{align}
 Hirschhorn,   Garvan,
and   Borwein \cite{Hirschhorn} proved that 
\begin{align}\label{3-8}
f_1^3= f_3 a(q^3)-3qf_9^3,
\end{align}
where 
\[
a(q)=\sum_{m,n=-\infty}^\infty q^{m^2+mn+n^2}.
\]
Thanks to \eqref{3-7}
 and \eqref{3-8}, 
\begin{align*}
	\sum_{n=0}^\infty a_{2^33^{-1} }(n)q^n=\frac{f_6}{f_3} a(q^6)-3q^2\frac{f_{18}^3}{f_3}, 
\end{align*}
which yields 
\begin{align}
\sum_{n=0}^\infty a_{2^33^{-1} }(3n)q^n
 &=\frac{1}{(q;q^2)_\infty }a(q^2) \succcurlyeq \sum_{n=0}^\infty
  q^n, \label{3-9}\\
\sum_{n=0}^\infty a_{2^33^{-1} }(3n+1)q^n&=0
,\label{3-10}\\
\sum_{n=0}^\infty a_{2^33^{-1} }(3n+2)q^n&=-3\frac{f_2}{f_1}\frac{f_6^3}{f_2}
=-\frac{3}{(q;q^2)_\infty }\sum_{n=0}^\infty c_3(n)q^{2n}\preccurlyeq - \sum_{n=0}^\infty q^n. \label{3-11}
\end{align}
Based on 
\eqref{3-9}--\eqref{3-11},
 we see that Theorem \ref{Th-2} is true when $(t_2,t_3,t_4)=(3,-1,0)$.
  \qed  
 
\begin{theorem}\label{Th-3}
	Define
	\[
	S_3:=\{(3,0,-2,0),(1,-1,-2,1),(2,-1,-2,0)\}
	\]
	If $(t_1,t_2,t_3,t_4)\in S_3$, then for $n\geq 0$,
	\begin{align*} 
		a_{1^{t_1}2^{t_2}3^{t_3}4^{t_4}}(3n) >0,\quad 
		a_{1^{t_1}2^{t_2}3^{t_3}4^{t_4}}(3n+1) <0,\quad a_{1^{t_1}2^{t_2}3^{t_3}4^{t_4}}(3n+2) =0.
	\end{align*}
\end{theorem}

\noindent{\it Proof.} In view of \eqref{1-1}, 
\begin{align}\label{3-12}
\sum_{n=0}^\infty a_{1^33^{-2} }(n)q^n=\frac{f_1^3 }{ f_3^2}.
\end{align}
Combining  \eqref{3-8} and \eqref{3-12}, we can prove that 
\begin{align*} 
\sum_{n=0}^\infty a_{1^33^{-2} }(3n)q^n&=\frac{a(q) }{ f_1}\succcurlyeq \sum_{n=0}^\infty q^n
,\\
\sum_{n=0}^\infty a_{1^33^{-2} }(3n+1)q^n&=-3\frac{f_3^3 }{ f_1^2}
=\frac{-3}{f_1}\sum_{n=0}^\infty c_3(n)q^n 
 \preccurlyeq - \sum_{n=0}^\infty q^n \\
\sum_{n=0}^\infty a_{1^33^{-2} }(3n+2)q^n&=0. 
\end{align*}
Thus, Theorem \ref{Th-3}
 is true when $(t_1,t_2,t_3,t_4)=(3,0,-2,0)$. 

By \eqref{1-1}, 
\begin{align}\label{3-14}
 \sum_{n=0}^\infty a_{1^12^{-1}3^{-2}4^1 }(n)q^n=\frac{f_1f_4}{f_2f_3^2}.
\end{align}
It follows from \cite[Corollary (ii), p. 49]{Berndt} that 
\begin{align}
	\frac{f_2^2}{f_1}=\frac{f_6f_9^2}{
	 f_3f_{18}}+q\frac{f_{18}^2}{f_9} .\label{3-15}
\end{align}
 Replacing $q$ by $-q$ in \eqref{3-15} yields 
\begin{align}\label{3-16}
 \frac{f_1f_4}{f_2}=\frac{f_3f_{12}f_{18}^5
 }{f_6^2f_9^2f_{36}^2}-q\frac{f_9f_{36}}{f_{18}}.
 \end{align}
Substituting \eqref{3-16} into \eqref{3-14}, we have
\[
 \sum_{n=0}^\infty a_{1^12^{-1}3^{-2}4^1 }(n)q^n=\frac{ f_{12}f_{18}^5
 }{f_3 f_6^2f_9^2f_{36}^2}-q\frac{f_9f_{36}}{f_3^2f_{18}},
\]
from which, we get 
 \begin{align*}
 \sum_{n=0}^\infty a_{1^12^{-1}3^{-2}4^1 }(3n)q^n&=\frac{f_4f_6^5}{f_1f_2^2f_3^2f_{12}^2}
 =\frac{1}{f_1f_2(q^2;q^4)_\infty}
 \left(\sum_{n=-\infty}^\infty q^{n^2}\right)\succcurlyeq 
 \sum_{n=0}^\infty q^n 
,\\
 \sum_{n=0}^\infty a_{1^12^{-1}3^{-2}4^1 }(3n+1)q^n&=-\frac{f_3f_{12}}{f_1^2f_{6}}
 =-\frac{1}{f_1(q,q^2;q^3)_\infty (q^6;q^{12})_\infty }\preccurlyeq 
 -\sum_{n=0}^\infty q^n
,\\
 \sum_{n=0}^\infty a_{1^12^{-1}3^{-2}4^1 }(3n+2)q^n&=0.
 \end{align*}
 Thus, Theorem \ref{Th-3}
 is true when $(t_1,t_2,t_3,t_4)=(1,-1,-2,1)$. 
 
In light of \eqref{1-1},
 \begin{align*}
\sum_{n=0}^\infty a_{1^22^{-1}3^{-2} }(n)q^n=\frac{f_1^2 }{f_2f_3^2}.
\end{align*}
Substituting  \eqref{3-3} into the above identity,   we can prove that 
 \begin{align}
\sum_{n=0}^\infty a_{1^22^{-1}3^{-2} }(3n)q^n&=\frac{f_3^2 }{f_1^2 f_6}
=\frac{1}{(q,q^2;q^3)_\infty^2 f_6} \succcurlyeq  \sum_{n=0}^\infty q^n ,\label{3-16-1}\\
\sum_{n=0}^\infty a_{1^22^{-1}3^{-2} }(3n+1)q^n&=-2\frac{f_6^2}{f_1f_2f_3}=-2\frac{1}{ f_1(q^2,q^3,q^4;q^6)_\infty }\preccurlyeq - \sum_{n=0}^\infty q^n
, \label{3-17}
\\
\sum_{n=0}^\infty a_{1^22^{-1}3^{-2} }(3n+2)q^n&=0.
\label{3-18}
 \end{align}
It  follows from \eqref{3-16-1}--\eqref{3-18}
 that Theorem \ref{Th-3} is true when $(t_1,t_2,t_3,t_4)=(2,-1,-2,0)$. 
  \qed

\begin{theorem}\label{Th-4}
	For $n\geq 0$, 
	\[
	a_{1^12^13^{-2} }(3n)>0, \quad a_{1^12^13^{-2} }(3n+1)<0,\quad 
	a_{1^12^13^{-2} }(3n+2)<0. 
	\]
\end{theorem}

\noindent{\it Proof.} By \eqref{1-1}, 
\begin{align}\label{3-19}
\sum_{n=0}^\infty a_{1^12^13^{-2} }(n)q^n=\frac{f_1f_2 }{ f_3^2}.
\end{align}
  Hirschhorn \cite[(14.3.1), p. 132]{Hirschhorn-1} proved that 
\begin{align}\label{3-20}
f_1f_2=\frac{f_6f_9^4}{f_3f_{18}^2}
-qf_9f_{18}-2q^2\frac{f_3f_{18}^4}{f_6f_9^2}.
\end{align}
Substituting \eqref{3-20}
 into \eqref{3-19},  we can prove  that 
\begin{align*}
\sum_{n=0}^\infty a_{1^12^13^{-2} }(3n)q^n&=\frac{f_2f_3^4 }{ f_1^3f_6^2}=\frac{1}{(q;q^2)_\infty
 (q,q^2;q^3)_\infty f_6^2}\sum_{n=0}^\infty c_3(n)q^n \succcurlyeq \sum_{n=0}^\infty q^n
,\\
\sum_{n=0}^\infty a_{1^12^13^{-2} }(3n+1)q^n&=-\frac{f_3f_6}{ f_1^2}
=-\frac{1}{(q,q^2;q^3)_\infty (q,q^2,q^3,q^4,q^5;q^6)_\infty }\preccurlyeq -\sum_{n=0}^\infty q^n,\\
\sum_{n=0}^\infty a_{1^12^13^{-2} }(3n+2)q^n&=-2\frac{f_6^4}{ f_1f_2f_3^2}= \frac{-2}{f_1f_2}\left(\sum_{n=0}^\infty c_2(n)q^{3n}\right)^2\preccurlyeq - \sum_{n=0}^\infty q^n.
\end{align*}
Hence, Theorem \ref{Th-4} is true. \qed 

	\section{Sign changes with period 4}
	
In this section, we  confirm several  conjectures 
	of  Bringmann, Han, Heim and Kane
	\cite{Bringmann-1}	on sign changes with period 4. 
	
\begin{theorem}\label{Th-4-1}
Define 
\[
S_4:=\{(-1,1,3,-3), (-3,8,1,-7)\}.
\]
	If $(t_1,t_2,t_3,t_4)\in S_4$, then for $n\geq 0$,
\begin{align*}
	a_{1^{t_1}2^{t_2}3^{t_3}4^{t_4}}(2n) >0,\quad 
	a_{1^{t_1}2^{t_2}3^{t_3}4^{t_4}}(4n+1) >0,\quad a_{1^{t_1}2^{t_2}3^{t_3}4^{t_4}}(4n+3) <0.
\end{align*}
\end{theorem}

\noindent{\it Proof.}
 In view of \eqref{1-1}, 
\begin{align}
\sum_{n=0}^\infty a_{1^{-1}2^13^{3}4^{-3} }(n)q^n=\frac{f_2f_3^3 }{ f_1f_4^3}. \label{4-2}
\end{align}
Hirschhorn \cite[(22.7.5), p. 196]{Hirschhorn-1} proved that 
\begin{align}
\frac{f_3^3}{f_1}=\frac{f_4^3f_6^2}{
 f_2^2f_{12}}+q\frac{f_{12}^3}{f_4}. \label{4-3}
\end{align}
Using \eqref{4-2} and \eqref{4-3}, we can prove that 
\begin{align} \label{4-4}
\sum_{n=0}^\infty a_{1^{-1}2^13^{3}4^{-3} }(2n)q^n=
\frac{f_3^2 }{ f_1f_6}=  \frac{1}{f_3f_6}
 \sum_{n=0}^\infty c_3(n)q^n \succcurlyeq \sum_{n=0}^\infty q^n 
\end{align}
and 
\begin{align}
\sum_{n=0}^\infty a_{1^{-1}2^13^{3}4^{-3} }(2n+1)q^n=
\frac{f_1f_6^3 }{ f_2^4}. \label{4-5}
\end{align}
Replacing $q$ by $-q$ in \eqref{4-5}, we arrive at 
\begin{align}
\sum_{n=0}^\infty (-1)^n a_{1^{-1}2^13^{3}4^{-3} }(2n+1)q^n=
\frac{f_6^3 }{ f_1f_2f_4}
 =\frac{1}{f_1f_4}\sum_{n=0}^\infty c_3(n)q^{2n}\succcurlyeq \sum_{n=0}^\infty q^n. 
\label{4-6}
\end{align}
It follows from 
 \eqref{4-4} and \eqref{4-6} that Theorem 
 \ref{Th-4-1} is true when $(t_1,t_2,t_3,t_4)=(-1,1,3,-3)$.

In light of  \eqref{1-1}, 
\begin{align}\label{4-6-1}
\sum_{n=0}^\infty a_{1^{-3}2^83^14^{-7} }(n)q^n=\frac{f_2^8f_3 }{ f_1^3f_4^7}.
\end{align}
Baruah  and  Ojah \cite[(3.49) and (3.50)]{Baruah} proved that 
\begin{align}
\frac{f_3}{f_1^3}=\frac{f_4^6f_6^3}{f_2^9f_{12}^2}+3q\frac{f_4^2f_6f_{12}^2}{f_2^7}.\label{4-7}
\end{align}
Substituting  \eqref{4-7} into \eqref{4-6-1}, 
 we can prove that 
\begin{align}
\sum_{n=0}^\infty a_{1^{-3}2^83^14^{-7} }(2n)q^n=\frac{f_3^3 }{ f_1f_2f_6^2} =\frac{1}{f_2f_6^2}\sum_{n=0}^\infty c_3(n)q^n \succcurlyeq \sum_{n=0}^\infty q^n \label{4-8}
\end{align}
and 
\begin{align}
\sum_{n=0}^\infty a_{1^{-3}2^83^14^{-7} }(2n+1)q^n=3
 \frac{f_1f_3f_6^2 }{ f_2^5}.\label{4-9}
\end{align}
Replacing $q$ by $-q$ in \eqref{4-9}, we arrive at 
\begin{align}\label{4-10}
\sum_{n=0}^\infty (-1)^n  a_{1^{-3}2^83^14^{-7} }(2n+1)q^n=3
\frac{ f_6^5 }{ f_1f_2^2 f_3f_4 f_{12}}=\frac{3}{ f_1  f_3f_4 f_6f_{12}}\left(\sum_{n=0}^\infty c_3(n)q^{2n}\right)^2\succcurlyeq   
\sum_{n=0}^\infty q^n.
\end{align}
It  follows from  \eqref{4-8}  and \eqref{4-10}
 that Theorem \ref{Th-4-1} is true when $(t_1,t_2,t_3,t_4)=(-3,8,1,-7)$.
    \qed

\begin{theorem}\label{Th-4-2}
	Define 
	\[
	S_5:=\{(-1,3,1,-3), (-1,4,1,-3),(-1,4,0,-4)\}.
	\]
	If $(t_1,t_2,t_3,t_4)\in S_5$, then for $n\geq 0$,
	\begin{align*}
		a_{1^{t_1}2^{t_2}3^{t_3}4^{t_4}}(4n)& >0,\quad 
		a_{1^{t_1}2^{t_2}3^{t_3}4^{t_4}}(4n+1) >0,\\ a_{1^{t_1}2^{t_2}3^{t_3}4^{t_4}}(4n+2)& <0,\quad 
		a_{1^{t_1}2^{t_2}3^{t_3}4^{t_4}}(4n+3) <0.
	\end{align*}
\end{theorem}

\noindent{\it Proof.}
In view of \eqref{1-1}, 
\begin{align}\label{4-11}
\sum_{n=0}^\infty a_{1^{-1}2^{t_2}3^14^{-3} }(n)q^n=\frac{f_2^{t_2}f_3}{ f_1f_4^3}.
\end{align}
Xia and Yao \cite[Lemma 2.6]{Xia-Yao} proved that 
\begin{align}
\frac{f_3}{f_1 }=\frac{f_4f_6f_{16}f_{24}^2}{
f_2^2f_8f_{12}f_{48}}+q\frac{ f_6f_8^2f_{48}}{
 f_2^2f_{16}f_{24}}.\label{4-12}
\end{align}
Substituting \eqref{4-12} into \eqref{4-11} yields 
\begin{align}\label{4-13}
\sum_{n=0}^\infty a_{1^{-1}2^{t_2}3^14^{-3} }(2n)q^n=\frac{f_1^{t_2-2}f_3f_8f_{12}^2}{ f_2^2f_4f_6f_{24}} 
\end{align}
and 
\begin{align}\label{4-14}
	\sum_{n=0}^\infty a_{1^{-1}2^{t_2}3^14^{-3} }(2n+1)q^n=\frac{f_1^{t_2-2}f_3f_4^2f_{24}}{ f_2^3f_8f_{12}}.
\end{align}
Replacing $q$ by $-q$ in \eqref{4-13}, we find that if $t_2\in\{3,4\}$, 
\begin{align}\label{4-15}
\sum_{n=0}^\infty (-1)^n a_{1^{-1}2^{t_2}3^14^{-3} }(2n)q^n=\frac{f_2^{3t_2-8}f_6^2f_8f_{12}}{ f_1^{t_2-2}f_3f_4^{t_2-1}f_{24}} \succcurlyeq \sum_{n=0}^\infty q^n .
\end{align}
Replacing $q$ by $-q$ in \eqref{4-14}, we deduce  that if $t_2\in\{3,4\}$, 
\begin{align}\label{4-16}
\sum_{n=0}^\infty (-1)^n a_{1^{-1}2^{t_2}3^14^{-3} }(2n+1)q^n 
= \frac{f_2^{3t_2-9}f_6^3f_{24}}{f_1^{t_2-2}f_3f_4^{t_2-4}f_8f_{12}^2}  \succcurlyeq \sum_{n=0}^\infty q^n .
\end{align}
By 
 \eqref{4-15} and \eqref{4-16}, we see that 
  Theorem \ref{Th-4-2}
   holds when $(t_1,t_2,t_3,t_4)\in\{(-1,3,1,-3), (-1,4,1,-3)\}$.

Thanks to \eqref{1-1}, 
\begin{align}
\sum_{n=0}^\infty   a_{1^{-1}2^{4}4^{-4} }(n)q^n 
=\frac{f_2^4}{f_1f_4^4}.\label{4-17}
\end{align}
	Xia and Yao \cite[Lemma 3.2]{Xia-Yao-2012} proved that 
\begin{align}\label{7-22}
	\frac{1}{f_1}=\frac{f_4^2}{f_2^3}(S(-q^2)+qT(-q^2))
\end{align}
and 
\begin{align}
	f_1=f_4(S(-q^2)-qT(-q^2)), \label{7-23}
\end{align}
where
\[
S(q):=\frac{1}{(q,q^4,q^7;q^8)_\infty},
 \qquad T(q):=\frac{1}{(q^3,q^4,q^5;q^8)_\infty}.
\]
Substituting \eqref{7-22} into \eqref{4-17},
 we obtain 
\begin{align}
	\sum_{n=0}^\infty    a_{1^{-1}2^{4}4^{-4} }(2n)q^n 
	=\frac{f_1}{f_2^2}S(-q)  \label{4-19}
\end{align}
and 
\begin{align}\label{4-20}
	\sum_{n=0}^\infty    a_{1^{-1}2^{4}4^{-4} }(2n+1)q^n 
	=\frac{f_1}{f_2^2}T(-q).
\end{align}
Replacing $q$ by $-q$ in \eqref{4-19} and 
\eqref{4-20}, we arrive at 
\begin{align}
	\sum_{n=0}^\infty (-1)^n   a_{1^{-1}2^{4}4^{-4} }(2n)q^n 
	=\frac{S(q)}{(q;q^2)_\infty f_4} \succcurlyeq  \sum_{n=0}^\infty q^n  \label{4-21}
\end{align}
and 
\begin{align}
	\sum_{n=0}^\infty  (-1)^n  a_{1^{-1}2^{4}4^{-4} }(2n+1)q^n 
	=\frac{T(q)}{(q;q^2)_\infty f_4}\succcurlyeq\sum_{n=0}^\infty q^n .\label{4-22}
\end{align}
Thus, Theorem \ref{Th-4-2} is true 
 when $(t_1,t_2,t_3,t_4)=(-1,4,0,-4)$.  This completes the proof.  \qed

 \begin{theorem} \label{Th-4-3}
 	For $n\geq 0$,
 	 \begin{align}\label{4-23}
 	 	a_{ 1^{-1}2^2 4^{-2} 5^1}(4n)>0,\quad a_{ 1^{-1}2^2 4^{-2} 5^1}(2n+1)>0, \quad a_{ 1^{-1}2^2 4^{-2} 5^1}(4n+2)=0. 
 	 \end{align}
 \end{theorem}

\noindent{\it Proof.}
 In view of \eqref{1-1}, 
\begin{align}
\sum_{n=0}^\infty a_{ 1^{-1}2^2 4^{-2} 5^1}(n)q^n=\frac{ f_2^2f_5}{ f_1f_4^2} . \label{4-24}
\end{align}
In \cite[(4.14)]{Xia-Yao-2012}, Xia and Yao proved that 
\begin{align}
\frac{f_5}{f_1}=\frac{f_8f_{20}^2}{f_2^2f_{40}}
+q\frac{f_4^3f_{10}f_{40}}{f_2^3f_8f_{20}}.
 \label{4-25}
\end{align}
Substituting \eqref{4-25} into \eqref{4-24}, we arrive at
\begin{align*}
	\sum_{n=0}^\infty a_{ 1^{-1}2^2 4^{-2} 5^1}(n)q^n=
	\frac{f_8f_{20}^2}{f_4^2f_{40}}+q\frac{f_4f_{10}f_{40}}{f_2f_8f_{20}},
\end{align*}
from which, we get 
\begin{align}
\sum_{n=0}^\infty a_{ 1^{-1}2^2 4^{-2} 5^1}(4n)q^n&=
\frac{f_2 f_{5}^2}{ f_1^2 f_{10}} 
 =\frac{ 1}{(q;q^2)_\infty 
 	f_5^3f_{10} }\sum_{n=0}^\infty c_5(n)q^n  \succcurlyeq \sum_{n=0}^\infty q^n 
, \label{4-26}\\
\sum_{n=0}^\infty a_{ 1^{-1}2^2 4^{-2} 5^1}(4n+2)q^n&=0 ,
\label{4-27}\\
\sum_{n=0}^\infty a_{ 1^{-1}2^2 4^{-2} 5^1}(2n+1)q^n&=\frac{ f_2f_5f_{20}}{ f_1f_4f_{10}} . \label{4-28}
\end{align}
Substituting \eqref{4-25} into \eqref{4-28}, we have 
\[
\sum_{n=0}^\infty a_{ 1^{-1}2^2 4^{-2} 5^1}(2n+1)q^n= 
\frac{ f_8f_{20}^3 }{ f_2f_4f_{10}f_{40}}+q\frac{ f_4^2f_{40} }{ f_2^2f_8} ,
\]
from which, we obtain 
\begin{align}\label{4-29}
\sum_{n=0}^\infty a_{ 1^{-1}2^2 4^{-2} 5^1}(4n+1)q^n=\frac{ f_4f_{10}^3 }{ f_1f_2f_5f_{20}} \succcurlyeq \sum_{n=0}^\infty q^n
\end{align}
and 
\begin{align}\label{4-30}
\sum_{n=0}^\infty a_{ 1^{-1}2^2 4^{-2} 5^1}(4n+3)q^n=\frac{ f_2^2f_{20} }{ f_1^2f_4} \succcurlyeq \sum_{n=0}^\infty q^n. 
\end{align}
Theorem \ref{Th-4-3} follows from \eqref{4-26}, \eqref{4-27},
 \eqref{4-29} and \eqref{4-30}. 
  The proof is complete.  \qed

\begin{theorem}\label{Th-4-4}
	Define 
	\begin{align*}
		S_6:=&\{(1,0,-1,-3), (1,1,-1,-3),(1,2,-1,-3),(1,3,-1,-3),(1,4,-1,-4),
		(2,0,0,-2), (2,1,0,-2),\\
		&\quad  (2,2,0,-2),(2,3,0,-3),(2,4,0,-3),(1,1,0,-3),
		 (1,2,0,-3), (1,3,0,-3),(1,4,0,-3)\}.
	\end{align*}
	If $(t_1,t_2,t_3,t_4)\in S_6$, then for $n\geq 0$,
	\begin{align*}
		a_{1^{t_1}2^{t_2}3^{t_3}4^{t_4}}(4n)& >0,\quad 
		a_{1^{t_1}2^{t_2}3^{t_3}4^{t_4}}(4n+1) <0,\\ a_{1^{t_1}2^{t_2}3^{t_3}4^{t_4}}(4n+2)& <0,\quad 
		a_{1^{t_1}2^{t_2}3^{t_3}4^{t_4}}(4n+3) >0.
	\end{align*}
\end{theorem}

\noindent{\it Proof.}
 By \eqref{1-1}, 
\begin{align}\label{4-31}
\sum_{n=0}^\infty a_{  1^12^{t_2}3^{-1}4^{-3}}(n)q^n=\frac{ f_1f_2^{t_2} }{ f_3f_4^3} . 
\end{align}
Xia and Yao \cite[Lemma 2.6]{Xia-Yao} proved that 
\begin{align}\label{4-32}
\frac{f_1}{f_3}=\frac{f_2f_{16}f_{24}^2
 }{f_6^2f_8f_{48}}-q\frac{f_2f_8^2f_{12}f_{48}}{f_4f_6^2f_{16}f_{24}}.
\end{align}
Substituting \eqref{4-32} into \eqref{4-31}, we get 
\[
\sum_{n=0}^\infty a_{  1^12^{t_2}3^{-1}4^{-3}}(n)q^n=
\frac{ f_2^{t_2+1}f_{16}f_{24}^2
}{ f_4^3f_6^2f_8f_{48}}-q\frac{ f_2^{t_2+1}f_8^2f_{12}f_{48} }{ f_4^4f_6^2f_{16}f_{24}} ,
\]
from which, we obtain 
\begin{align}\label{4-33}
\sum_{n=0}^\infty a_{  1^12^{t_2}3^{-1}4^{-3}}(2n)q^n=\frac{ f_1^{t_2+1}f_8f_{12}^2
  }{ f_2^3f_3^2f_4f_{24}} 
\end{align}
and 
\begin{align}\label{4-34}
\sum_{n=0}^\infty a_{  1^12^{t_2}3^{-1}4^{-3}}(2n+1)q^n=-\frac{ f_1^{t_2+1}f_4^2f_6f_{24} }{ f_2^4f_3^2f_8f_{12}} .
\end{align}
Replacing $q$ by $-q$ in \eqref{4-33}
 and \eqref{4-34}, we see that if $0\leq t_2\leq 3$, 
\begin{align}\label{4-35}
\sum_{n=0}^\infty (-1)^n a_{  1^12^{t_2}3^{-1}4^{-3}}(2n)q^n=\frac{ f_2^{3t_2}f_3^2f_8f_{12}^4
}{ f_1^{t_2+1}f_4^{t_2+2}f_6^6f_{24}} \succcurlyeq
 \sum_{n=0}^\infty q^n  
\end{align}
and 
\begin{align}\label{4-36}
\sum_{n=0}^\infty (-1)^n  a_{  1^12^{t_2}3^{-1}4^{-3}}(2n+1)q^n=-\frac{  f_2^{3t_2-1}f_3^2f_{12}f_{24} }{ f_1^{t_2+1}f_4^{t_2-1}f_6^5f_8}  \preccurlyeq - \sum_{n=0}^\infty q^n.
\end{align}

Using the same method, we can prove that 
\begin{align}\label{4-37}
\sum_{n=0}^\infty (-1)^n a_{  1^12^{4}3^{-1}4^{-4}}(2n)q^n= 
\left(\sum_{n=-\infty }^\infty q^{n^2}\right)^2 \left(\sum_{n=-\infty}^\infty q^{n(3n+1)/2}\right)
\frac{f_8f_{12}^4}{f_4^2f_6^5f_{24}} \succcurlyeq
\sum_{n=0}^\infty q^n
\end{align}
and 
\begin{align}\label{4-38}
\sum_{n=0}^\infty (-1)^n a_{  1^12^{4}3^{-1}4^{-4}}(2n+1)q^n= 
-\left(\sum_{n=-\infty}^\infty q^{3n^2+2n}\right)^2\left(\sum_{n=0}^\infty c_2(n)q^{n}\right)^3 
\frac{f_{24}}{f_4f_6^3f_8f_{12}} \preccurlyeq 
-\sum_{n=0}^\infty q^n. 
\end{align}

  By \eqref{1-1}, 
\begin{align} \label{4-39}
\sum_{n=0}^\infty a_{ 1^22^{t_2}4^{-2} }(n)q^n=\frac{ f_1^2f_2^{t_2}}{  f_4^2} .
\end{align}
It follows from \cite[Lemma 2.2]{Xia-Yao} that 
\begin{align}\label{4-40}
	f_1^2=\frac{f_2f_8^5}{f_4^2f_{16}^2}
	-2q
	\frac{f_2f_{16}^2}{f_8}
\end{align}
and 
\begin{align}\label{m-1}
	\frac{1}{f_1^2}
	=\frac{ f_8^5}{f_2^5f_{16}^2}
	+2q
	\frac{f_4^2 f_{16}^2}{f_2^5 f_8}. 
\end{align}

 Substituting \eqref{4-40} into \eqref{4-39} yields 
\begin{align*}
	\sum_{n=0}^\infty a_{ 1^22^{t_2}4^{-2} }(n)q^n= \frac{f_2^{t_2+1}f_8^5}{f_4^4f_{16}^2}-2q\frac{f_2^{t_2+1}
	f_{16}^2}{f_4^2f_8},
\end{align*}
from which, we get 
  \begin{align}\label{4-41}
  	\sum_{n=0}^\infty a_{ 1^22^{t_2}4^{-2} }(2n)q^n= \frac{f_1^{t_2+1}f_4^5}{f_2^4f_{8}^2}  
  \end{align}
  and 
  \begin{align}\label{4-42}
  	\sum_{n=0}^\infty a_{ 1^22^{t_2}4^{-2} }(2n+1)q^n=  -2 \frac{f_1^{t_2+1}
  		f_{8}^2}{f_2^2f_4}.
  \end{align}
  Replacing $q$ by $-q$ in \eqref{4-41} and \eqref{4-42},
   we see that for $0\leq t_2\leq 2$, 
    \begin{align}
  	\sum_{n=0}^\infty (-1)^n a_{ 1^22^{t_2}4^{-2} }(2n)q^n= \frac{ f_2^{3t_2-1} f_4^{4-t_2} }{f_1^{t_2+1}f_{8}^2}  \succcurlyeq \sum_{n=0}^\infty q^n \label{4-43}
  \end{align}  
and 
    \begin{align}
	\sum_{n=0}^\infty (-1)^n a_{ 1^22^{t_2}4^{-2} }(2n+1)q^n= -2 \frac{ f_2^{3t_2+1} f_8^{2} }{f_1^{t_2+1}f_{4}^{t_2+2}}   \preccurlyeq - \sum_{n=0}^\infty q^n. \label{4-44}
\end{align}

Utilizing \eqref{4-40} 
 and the same method
 as in the proof of  \eqref{4-43} and \eqref{4-44}, we can prove that  if $t_2\in \{3,4\}$, then 
  \begin{align}
 	\sum_{n=0}^\infty (-1)^n a_{ 1^22^{t_2}4^{-3} }(2n)q^n= \frac{ f_2^{3t_2-2}f_4^{4-t_2} }{f_1^{t_2+1} f_{8}^2}  \succcurlyeq \sum_{n=0}^\infty q^n \label{4-45}
 \end{align}  
 and  \begin{align}
 	\sum_{n=0}^\infty (-1)^n a_{ 1^22^{t_2}4^{-3} }(2n+1)q^n=   -2 \frac{ f_2^{3t_2 } f_8^{2} }{f_1^{t_2+1}f_{4}^{t_2+2}} \preccurlyeq -\sum_{n=0}^\infty q^n .\label{4-46}
 \end{align}  
 The generating 
  function of $a_{  1^12^{t_2} 4^{-3}}(n)$ is 
 \begin{align}\label{g-1}
 	\sum_{n=0}^\infty a_{  1^12^{t_2} 4^{-3}}(n)q^n=\frac{ f_1f_2^{t_2} }{  f_4^3} . 
 \end{align}
 Substituting \eqref{7-23}
  into \eqref{g-1} yields 
   \begin{align*}
  	\sum_{n=0}^\infty a_{  1^12^{t_2} 4^{-3}}(n)q^n=\frac{  f_2^{t_2} }{  f_4^2 }(S(-q^2)-qT(-q^2)),  
  \end{align*}
 from which,  we arrive at
  \begin{align}\label{g-2}
 	\sum_{n=0}^\infty a_{  1^12^{t_2} 4^{-3}}(2n)q^n=\frac{  f_1^{t_2} }{  f_2^2}S(-q)  
 \end{align}
 and 
   \begin{align}\label{g-3}
 	\sum_{n=0}^\infty a_{  1^12^{t_2} 4^{-3}}(2n+1)q^n=-\frac{  f_1^{t_2} }{  f_2^2}T(-q)  .
 \end{align}
 Replacing $q$ by $-q$ in \eqref{g-2} and \eqref{g-3}, we deduce that 
  for $1\leq t_2\leq 4$, 
   \begin{align}\label{g-4}
 	\sum_{n=0}^\infty (-1)^n  a_{  1^12^{t_2} 4^{-3}}(2n)q^n=\frac{  f_2^{3t_2-2} }{  f_1^{t_2}f_4^{t_2}}S(q)  \succcurlyeq 
 	\sum_{n=0}^\infty q^n
 \end{align}
 and 
    \begin{align}\label{g-5}
 	\sum_{n=0}^\infty (-1)^n  a_{  1^12^{t_2} 4^{-3}}(2n+1)q^n=-\frac{  f_2^{3t_2-2} }{  f_1^{t_2}f_4^{t_2}}T(q)  \preccurlyeq
 	-\sum_{n=0}^\infty q^n.
 \end{align}
Theorem \ref{Th-4-4} follows from \eqref{4-35}--\eqref{4-38},  \eqref{4-43}-\eqref{4-46},  and \eqref{g-4}--\eqref{g-5}. This completes the proof. \qed

\begin{theorem}\label{Th-4-5}
	For $n\geq 0$, 
	\[
	 a_{ 1^13^14^{-3} }(4n)>0,\quad  a_{ 1^13^14^{-3} }(4n+2)<0,
	 \quad  a_{ 1^13^14^{-3} }(2n+1)<0.
	\]
	\end{theorem}
	
\noindent{\it Proof.}
 In view of \eqref{1-1}, 
\begin{align}\label{4-47}
\sum_{n=0}^\infty a_{ 1^13^14^{-3} }(n)q^n=\frac{ f_1 f_3}{  f_4^3} .
\end{align}
Xia and Yao \cite[Lemma 3.2]{Xia-Yao-2012} proved that 
\begin{align}\label{4-48}
f_1f_3=\frac{f_2f_8^2f_{12}^4}{
 f_4^2f_6f_{24}^2}-q\frac{f_4^4f_6f_{24}^2}{
 f_2f_8^2f_{12}^2}.
\end{align}
Substituting \eqref{4-48} into \eqref{4-47},  we get 
\[
\sum_{n=0}^\infty a_{ 1^13^14^{-3} }(n)q^n=\frac{f_2f_8^2f_{12}^4}{
	f_4^5f_6f_{24}^2}-q\frac{f_4f_6f_{24}^2}{
	f_2f_8^2f_{12}^2},
\]
which implies that 
\begin{align}\label{4-49}
\sum_{n=0}^\infty a_{ 1^13^14^{-3} }(2n)q^n=\frac{f_1f_4^2f_6^4}{f_2^5f_3f_{12}^2}
\end{align}
and 
\begin{align}\label{4-50}
\sum_{n=0}^\infty a_{ 1^13^14^{-3} }(2n+1)q^n=-\frac{f_{12}}{f_2f_4f_6} \sum_{n=-\infty}^\infty q^{3n^2+2n}\preccurlyeq -\sum_{n=0}^\infty q^n.
\end{align}
Replacing $q$ by $-q$ in \eqref{4-49} yields 
\begin{align}\label{4-51}
\sum_{n=0}^\infty (-1)^n  a_{ 1^13^14^{-3} }(2n)q^n=\frac{f_3f_4f_6}{f_1f_2^2f_{12}} \succcurlyeq\sum_{n=0}^\infty q^n.
\end{align}
Theorem 
\ref{Th-4-5} follows from \eqref{4-50} and \eqref{4-51}.
 This completes the proof. \qed

\begin{theorem}\label{Th-4-6}
	For $n\geq 0$, 
	\[
	a_{ 1^22^{-1}4^{-1} }(4n)>0,\quad  a_{ 1^22^{-1}4^{-1} }(4n+1)<0,
	\quad a_{ 1^22^{-1}4^{-1} }(4n+2)=a_{ 1^22^{-1}4^{-1} }(4n+3)=0.
	\]
\end{theorem}

\noindent{\it Proof}.
 By \eqref{1-1}, 
\begin{align}\label{4-52}
\sum_{n=0}^\infty a_{ 1^22^{-1}4^{-1} }(n)q^n=\frac{ f_1^2}{ f_2 f_4} .
\end{align}
 Substituting \eqref{4-40} into \eqref{4-52}, we get
  \[
  \sum_{n=0}^\infty a_{ 1^22^{-1}4^{-1} }(n)q^n= \frac{f_8^5}{
  f_4^3f_{16}^2}-2q\frac{f_{16}^2}{f_4f_8},
  \]
  from which, we deduce that 
    \begin{align}
  \sum_{n=0}^\infty a_{ 1^22^{-1}4^{-1} }(4n)q^n&= \frac{f_2^5}{
  	f_1^3f_{4}^2}\succcurlyeq \sum_{n=0}^\infty q^n ,\label{4-53}\\
  	\sum_{n=0}^\infty a_{ 1^22^{-1}4^{-1} }(4n+1)q^n&= -2\frac{f_4^2}{
  		f_1f_2}\preccurlyeq - \sum_{n=0}^\infty q^n ,
  		\label{4-54}\\
  		\sum_{n=0}^\infty a_{ 1^22^{-1}4^{-1} }(4n+2)q^n&= \sum_{n=0}^\infty a_{ 1^22^{-1}4^{-1} }(4n+3)q^n=0.
  		\label{4-55}
  \end{align}
  Theorem \ref{Th-4-6} follows from \eqref{4-53}--\eqref{4-55}.
   This completes the proof. \qed

    \begin{theorem}\label{Th-4-7}
  	For $n\geq 0$, 
  	\[
  	a_{1^23^{2}4^{-3}}(4n)>0,\
 a_{1^23^{2}4^{-3}}(4n+1)<0,\  a_{1^23^{2}4^{-3}}(4n+2)<0,
  \  a_{1^23^{2}4^{-3}}(4n+3)=0.
  	\]
  \end{theorem}
  
  \noindent{\it Proof.} By \eqref{1-1},
  \begin{align}\label{4-56}
  	  \sum_{n=0}^\infty 
  	  	a_{1^23^{2}4^{-3}}( n)q^n=\frac{f_1^2f_3^2}{f_4^3}. 
  \end{align}
   Substituting \eqref{4-40} into \eqref{4-56}, 
    then picking out those
    terms in which the power of  $q $ is congruent to 0  modulo 2
      and 
    replacing $q^{2}$ by $q$, we obtain 
   \begin{align}\label{4-57}
   	\sum_{n=0}^\infty 
   	a_{1^23^{2}4^{-3}}( 2n)q^n=\frac{f_1f_3f_4^5f_{12}^5}{f_2^5f_6^2f_8^2f_{24}^2}
   	 +4q^2\frac{f_1f_3f_8^2f_{24}^2}{f_2^3f_4f_{12}}. 
   \end{align}
   Replacing $q$ by $-q$ in \eqref{4-57}, we get 
      \begin{align}\label{4-58}
   	\sum_{n=0}^\infty  (-1)^n
   	a_{1^23^{2}4^{-3}}( 2n)q^n=\frac{ f_4^4f_6f_{12}^4 }{f_1f_2^2f_3f_8^2f_{24}^2}
   	+4q^2\frac{f_6^3f_8^2f_{24}^2}{f_1f_3f_4^2f_{12}^2}
   	\succcurlyeq \sum_{n=0}^\infty q^n. 
   \end{align}
      Substituting \eqref{4-48} into \eqref{4-56} and picking out those
   terms in which the power of  $q $ is congruent to 1  modulo 2,
    then dividing by $q$
    and 
   replacing $q^{2}$ by $q$, we have 
     \begin{align*}
   	\sum_{n=0}^\infty 
   	a_{1^23^{2}4^{-3}}( 2n+1)q^n=-2\frac{f_6^2}{f_2},
   \end{align*}
   from which, we deduce that 
       \begin{align}\label{4-59}
   	\sum_{n=0}^\infty 
   	a_{1^23^{2}4^{-3}}( 4n+1)q^n=-2\frac{f_3^2}{f_1} \preccurlyeq -\sum_{n=0}^\infty q^n 
   \end{align}
   and 
       \begin{align} \label{4-60}
   	\sum_{n=0}^\infty 
   	a_{1^23^{2}4^{-3}}( 4n+3)q^n=0.
   \end{align}
   Theorem \ref{Th-4-7}
    follows from \eqref{4-58}--\eqref{4-60}.
     This completes the proof. \qed

  \begin{theorem}\label{Th-4-8}
  	Define 
  	\begin{align*}
  		S_7:=&\{( 3,-1,-1,-2),
  		(4,-1,0,-3), (4,0,0,-6), (4,1,0,-9)\}.
  	\end{align*}
  	If $(t_1,t_2,t_3,t_4)\in S_7$, then for $n\geq 0$,
  	\begin{align*}
  		a_{1^{t_1}2^{t_2}3^{t_3}4^{t_4}}(2n)  >0,\quad 
  		a_{1^{t_1}2^{t_2}3^{t_3}4^{t_4}}(4n+1) <0,\quad 
  		a_{1^{t_1}2^{t_2}3^{t_3}4^{t_4}}(4n+3) >0.
  	\end{align*}
  \end{theorem}

  \noindent{\it Proof.} By \eqref{1-1},
  \begin{align}
  	\sum_{n=0}^\infty a_{1^32^{-1}3^{-1}4^{-2}}(n)q^n =
  	\frac{f_1^3}{f_2f_3f_4^2}. \label{4-62}
  \end{align}
  Xia and Yao \cite[Lemma 2.4]{Xia-Yao}
   proved that 
   \begin{align}
   	\frac{f_1^3}{f_3}=\frac{f_4^3}{f_{12}}-3q\frac{ f_2^2f_{12}^3
   	}{f_4f_6^2}.\label{4-63}
   \end{align}
   Substituting \eqref{4-63} into \eqref{4-62}, we get 
    \begin{align*}
   	\sum_{n=0}^\infty a_{1^32^{-1}3^{-1}4^{-2}}(n)q^n =
   	\frac{f_4 }{f_2f_{12}}-3q\frac{ f_2 f_{12}^3
   	}{f_4^3f_6^2},
   \end{align*}
   from which, we deduce that 
    \begin{align}\label{4-64}
   	\sum_{n=0}^\infty a_{1^32^{-1}3^{-1}4^{-2}}(2n)q^n =
   	\frac{f_2 }{f_1f_{6}}\succcurlyeq \sum_{n=0}^\infty q^n
   \end{align}
   and 
       \begin{align}\label{4-65}
   	\sum_{n=0}^\infty a_{1^32^{-1}3^{-1}4^{-2}}(2n+1)q^n =
 -3\frac{ f_1 f_{6}^3
   	}{f_2^3f_3^2}. 
   \end{align}
   Replacing $q$ by $-q$ in \eqref{4-65} yields 
          \begin{align}\label{4-66}
   	\sum_{n=0}^\infty  (-1)^n a_{1^32^{-1}3^{-1}4^{-2}}(2n+1)q^n =
   	-3\frac{ f_3^2f_{12}^2
   	}{f_1f_4f_6^3}\preccurlyeq -\sum_{n=0}^\infty q^n. 
   \end{align}
   It follows from \eqref{4-64}
    and \eqref{4-66} that  Theorem \ref{Th-4-8}
     is true when $(t_1,t_2,t_3,t_4)=( 3,-1,-1,-2)$. 
     
     In view of \eqref{1-1},
     \begin{align}\label{4-67}
     	\sum_{n=0}^\infty
     	a_{1^42^{-1}4^{-3}}(n) q^n =\frac{f_1^4}{f_2f_4^3}.
     \end{align}
 Xia and Yao \cite[Lemma 2.3]{Xia-Yao}
  proved that 
  \begin{align}
  	f_1^4 =\frac{f_4^{10}}{f_2^2f_8^4}
  	-4q\frac{f_2^2f_8^4}{f_4^2}.\label{4-68}
  \end{align}
  Substituting \eqref{4-68}
   into \eqref{4-67}, we get 
      \begin{align*} 
   	\sum_{n=0}^\infty
   	a_{1^42^{-1}4^{-3}}(n) q^n =\frac{f_4^{7}}{f_2^3f_8^4}
   	-4q\frac{f_2f_8^4}{f_4^5}.
   \end{align*}
     From the above identity, we have 
           \begin{align} \label{4-69}
     	\sum_{n=0}^\infty
     	a_{1^42^{-1}4^{-3}}(2n) q^n =\frac{f_2^{7}}{f_1^3f_4^4}
     =\frac{1}{f_4^2}\left(\sum_{n=-\infty}^\infty q^{n^2}\right)
     \left(\sum_{n=0}^\infty c_2(n) q^n\right)	\succcurlyeq \sum_{n=0}^\infty q^n
     \end{align}
     and 
          \begin{align*} 
     	\sum_{n=0}^\infty
     	a_{1^42^{-1}4^{-3}}(2n+1) q^n = 
     	-4 \frac{f_1f_4^4}{f_2^5}.
     \end{align*}
     Replacing $q$ by $-q$ in the above identity yields 
          \begin{align} 
     	\sum_{n=0}^\infty (-1)^n 
     	a_{1^42^{-1}4^{-3}}(2n+1) q^n = 
     	-4 \frac{ f_4^3}{f_1f_2^2} \preccurlyeq - \sum_{n=0}^\infty 
     	 q^n. \label{4-70}
     \end{align}
      It follows from \eqref{4-69}
     and \eqref{4-70} that  Theorem \ref{Th-4-8}
     is true when $(t_1,t_2,t_3,t_4)=( 4,-1,0,-3)$. 
     
     Using the same method for proving \eqref{4-69}
     and \eqref{4-70}, one can prove that 
              \begin{align}  
     	\sum_{n=0}^\infty
     	a_{1^44^{-6}}(2n) q^n& =\frac{f_2^{4}}{f_1^2f_4^4}
     	\succcurlyeq \sum_{n=0}^\infty q^n
   ,\label{4-71}\\
     	\sum_{n=0}^\infty (-1)^n 
     	a_{1^4 4^{-6}}(2n+1) q^n &= 
     	-4 \frac{ f_4^2}{f_1^2f_2^2} \preccurlyeq - \sum_{n=0}^\infty 
     	q^n  \label{4-72}
     \end{align}
     and 
           \begin{align}  
     	\sum_{n=0}^\infty
     	a_{1^42^1 4^{-9}}(2n) q^n& =\frac{f_2}{f_1f_4^4}
     	\succcurlyeq \sum_{n=0}^\infty q^n
     	,\label{4-73}\\
     	\sum_{n=0}^\infty (-1)^n 
     	a_{1^42^1 4^{-9}}(2n+1) q^n &= 
     	-4 \frac{ f_4}{f_1^3f_2^2} \preccurlyeq - \sum_{n=0}^\infty 
     	q^n.  \label{4-74}
     \end{align}
           It follows from \eqref{4-71}--\eqref{4-74} that  Theorem \ref{Th-4-8}
     is true when $(t_1,t_2,t_3,t_4)=( 4,0,0,-6)$ or  $(t_1,t_2,t_3,t_4)=( 4,1,0,-9)$. This completes 
      the proof of Theorem \ref{Th-4-8}. \qed

  \begin{theorem}\label{Th-4-9}
  	For $n\geq 0$, 
  	\[
  	 a_{1^33^{-1}4^{-4}}(4n)>0,\
  	   a_{1^33^{-1}4^{-4}}(4n+1)<0,\  a_{1^33^{-1}4^{-4}}(4n+2)=0,
  	   \  a_{1^33^{-1}4^{-4}}(4n+3)>0.
  	\]
  	\end{theorem}
  	
  	\noindent{\it Proof.}
  	 The generating function of $a_{1^33^{-1}4^{-4}}(n)$
  	  is 
  	  \[
  	   \sum_{n=0}^\infty
  	  a_{1^33^{-1}4^{-4}}(n)q^n =\frac{f_1^3}{
  	  	f_3f_4^4}.
  	  \]
  Substituting \eqref{4-63} into the above identity, we get 
  \[
  \sum_{n=0}^\infty
   a_{1^33^{-1}4^{-4}}(n)q^n  =\frac{1}{f_4f_{12}}-3q\frac{ f_2^2f_{12}^3
     }{f_4^5f_6^2},
  \]
  which yields 
    \begin{align}
  \sum_{n=0}^\infty
a_{1^33^{-1}4^{-4}}(4n)q^n &=\frac{1}{f_1f_{3}} \succcurlyeq 
  \sum_{n=0}^\infty q^n 
,\label{4-75}\\
  \sum_{n=0}^\infty
a_{1^33^{-1}4^{-4}}(4n+2)q^n &=0
, \label{4-76}\\
  \sum_{n=0}^\infty
a_{1^33^{-1}4^{-4}}(2n+1)q^n &= -3 \frac{ f_1^2f_{6}^3
  }{f_2^5f_3^2}.\label{4-77}
  \end{align}
  Replacing $q$ by $-q$ in \eqref{4-77}
   yields 
    \begin{align}\label{4-78}
  \sum_{n=0}^\infty (-1)^n 
  a_{1^33^{-1}4^{-4}}(2n+1)q^n = -3 \frac{f_2f_3^2f_{12}^2
  }{f_1^2f_4^2f_6^3} \preccurlyeq -  \sum_{n=0}^\infty q^n .
\end{align}
Theorem \ref{Th-4-9} follows from
 \eqref{4-75}, \eqref{4-76} and \eqref{4-78}.
  This completes the proof. \qed

    \begin{theorem}\label{Th-4-10}
  	For $n\geq 0$, 
  	\[
  	 a_{1^42^{-2}4^{-1}}(2n)>0,\
   a_{1^42^{-2}4^{-1}}(4n+1)<0,\   a_{1^42^{-2}4^{-1}}(4n+3)=0 .
  	\]
  \end{theorem}
  
  \noindent{\it Proof.}
   The generating function of  $a_{1^42^{-2}4^{-1}}(n)$
    is 
      \begin{align}\label{4-79}
  \sum_{n=0}^\infty   
  a_{1^42^{-2}4^{-1}}(n)q^n =  \frac{ f_1^4
  }{f_2^2f_4}.
  \end{align}
  Substituting \eqref{4-68} into \eqref{4-79},
   we deduce that 
       \begin{align}
  \sum_{n=0}^\infty   
  a_{1^42^{-2}4^{-1}}(2n)q^n &=  \frac{1}{f_2}\left(
  \sum_{n=-\infty}^\infty q^{n^2}\right)^2 \succcurlyeq \sum_{n=0}^\infty q^n 
, \label{4-80}\\
  \sum_{n=0}^\infty   
  a_{1^42^{-2}4^{-1}}(4n+1)q^n &= -4 \frac{ f_2^4
  }{f_1^3}\preccurlyeq - \sum_{n=0}^\infty q^n \label{4-81}
  \end{align}
  and 
   \begin{align}
  \sum_{n=0}^\infty   
  a_{1^42^{-2}4^{-1}}(4n+3)q^n = 0. \label{4-82}
\end{align}
Theorem \ref{Th-4-10}
 follows from \eqref{4-80}--\eqref{4-82}. This completes the proof. \qed

	\section{Sign changes with period 5}
	
This section is devoted to the proofs of  several  conjectures 
of  Bringmann, Han, Heim and Kane
\cite{Bringmann-1}	on sign changes with period 5. 
  
    \begin{theorem}\label{Th-5-1}
	For $n\geq 0$, 
	\begin{align*}
a_{1^{-1}2^{2}5^{-1}}(5n)&>0,\
a_{1^{-1}2^{2}5^{-1}}(5n+1)>0,\   a_{1^{-1}2^{2}5^{-1}}(5n+2)=0 ,\\ a_{1^{-1}2^{2}5^{-1}}(5n+3)&>0,\  a_{1^{-1}2^{2}5^{-1}}(5n+4)=0 .
	\end{align*}
\end{theorem}

\noindent{\it Proof}. 
 The generating function 
  of  $  a_{1^{-1}2^{2}5^{-1}}(n)$ is 
  \begin{align}\label{5-1}
  \sum_{n=0}^\infty   
  a_{1^{-1}2^{2}5^{-1}}(n)q^n = \frac{f_2^2}{f_1f_5}.
  \end{align}
  It follows from \cite[Corollary (ii), p. 49]{Berndt}
\[
\frac{f_2^2}{f_1}=f(q^{10},q^{15})+qf(q^5,q^{20})+q^3\frac{f_{50}^2}{f_{25}},
\]
where $f(a,b)=(-a,-b,ab;ab)_\infty$. Substituting 
the above identity into \eqref{5-1}, we get 
   \begin{align*} 
 	\sum_{n=0}^\infty   
 	a_{1^{-1}2^{2}5^{-1}}(5n)q^n& = \frac{(-q^{2},-q^{3};q^5)_\infty }{(q,q^2,q^3,q^4;q^5)_\infty  } \succcurlyeq  \sum_{n=0}^\infty q^n ,\\
 	\sum_{n=0}^\infty   
 	a_{1^{-1}2^{2}5^{-1}}(5n+1)q^n &= \frac{(-q,-q^{4};q^5)_{\infty}}{(q,q^2,q^3,q^4;q^5)_\infty  } \succcurlyeq \sum_{n=0}^\infty q^n ,\\
 	\sum_{n=0}^\infty   
 	a_{1^{-1}2^{2}5^{-1}}(5n+2)q^n& = 0,\\
 	\sum_{n=0}^\infty   
 	a_{1^{-1}2^{2}5^{-1}}(5n+3)q^n &= \frac{f_{10}^2}{f_1f_5 }\succcurlyeq \sum_{n=0}^\infty q^n ,\\
 		\sum_{n=0}^\infty   
 	a_{1^{-1}2^{2}5^{-1}}(5n+4)q^n& = 0.
 \end{align*}
 Theorem \ref{Th-5-1} follows from the above five identities. The proof
  is complete. \qed

      \begin{theorem}\label{Th-5-2}
  	For $n\geq 0$, 
  	\begin{align*}
  & a_{1^{-1}2^{3}4^{-1}5^{-2}}(5n)>0,\
  	 a_{1^{-1}2^{3}4^{-1}5^{-2}}(5n+1)>0,\    a_{1^{-1}2^{3}4^{-1}5^{-2}}(5n+2)<0 ,\\ & a_{1^{-1}2^{3}4^{-1}5^{-2}}(5n+3)=0,\   a_{1^{-1}2^{3}4^{-1}5^{-2}}(5n+4)=0 .
  	\end{align*}
  \end{theorem}
  
  \noindent{\it Proof.} The generating 
   function of  $ a_{1^{-1}2^{3}4^{-1}5^{-2}}(n)$  is
           \begin{align}\label{5-4}
   \sum_{n=0}^\infty   
   a_{1^{-1}2^{3}4^{-1}5^{-2}}(n)q^n = \frac{f_2^3}{f_1f_4f_5^2}.
   \end{align}
 It follows from \cite[(8.1.1), p. 85]{Hirschhorn-1} that 
\begin{align}\label{5-5}
	f_1=f_{25}\left(\frac{1}{R(q^5)}-q-q^2R(q^5)\right),
\end{align}
where 
\[
R(q)=\frac{(q,q^4;q^5)_\infty }{(q^2,q^3;q^5)_\infty }. 
\]
  Replacing $q$ by $-q$ in \eqref{5-5} yields  
\begin{align}\label{5-6}
  \frac{f_2^3}{f_1f_4}= \frac{f_{50}^3}{f_{25}f_{100}}
   \left(\frac{1}{R(-q^5)}+q-q^2R(-q^5)\right).
  \end{align}
  Substituting \eqref{5-6} into \eqref{5-4}, we can prove Theorem
   \ref{Th-5-2}. \qed

   \begin{theorem}\label{Th-5-3}
   	For $n\geq 0$, 
   	\[
   a_{ 2^{1} 5^{-1}}(5n)>0,\
   	a_{ 2^{1} 5^{-1}}(5n+1)=0,\    a_{ 2^{1} 5^{-1}}(5n+2)<0 ,\  a_{ 2^{1} 5^{-1}}(5n+3)=0,\   a_{ 2^{1} 5^{-1}}(5n+4)<0 .
   	\]
   \end{theorem}
   
   \noindent{\it Proof.} The generating function 
     of $ a_{ 2^{1} 5^{-1}}(n)$ is 
        \[
  \sum_{n=0}^\infty   
  a_{ 2^{1} 5^{-1}}(n)q^n = \frac{f_2 }{ f_5 }.
  \]
  Based on \eqref{5-5}
   and the above identity, we can prove Theorem \ref{Th-5-3}.
    \qed 
  
       \begin{theorem}\label{Th-5-4}
    	For $n\geq 0$, 
    	\begin{align*}
    	 a_{1^{2}2^{-1}5^{-2}}(5n)&>0,\
    	 a_{1^{2}2^{-1}5^{-2}}(5n+1)<0,\     a_{1^{2}2^{-1}5^{-2}}(5n+2)=0 ,\\ a_{1^{2}2^{-1}5^{-2}}(5n+3)&=0,\   a_{1^{2}2^{-1}5^{-2}}(5n+4)>0 .
    	\end{align*}
    \end{theorem}
    
    \noindent{\it Proof.}
     The generating function of $a_{1^{2}2^{-1}5^{-2}}(n)$
      is 
           \begin{align}\label{5-7}
  \sum_{n=0}^\infty   
  a_{1^{2}2^{-1}5^{-2}}(n)q^n = \frac{f_1^2}{f_2f_5^2}.
  \end{align}
  It follows from \cite[Corollary (ii),  p. 49]{Berndt}
   that 
   \[
   \frac{f_1^2}{f_2}=\frac{f_{25}^2}{f_{50}}
   -2qf(-q^{15},-q^{35})+2q^4f(-q^5,-q^{45}). 
   \]
Utilizing  \eqref{5-7}
  and the above identity, we can prove Theorem \ref{Th-5-4}.
  \qed

%
%
  
  	\section{Sign changes with period 6}
  	
  	The objective of this section is to 
  	 confirm several conjectures
  	  of    Bringmann, Han, Heim and Kane
  	  \cite{Bringmann-1} on sign changes with period 6.
  	
   \begin{theorem}\label{Th-6-1}
  	Define 
  	\begin{align*}
  		S_8:=&\{( 1,-4,-3,4),
  		(3,-3,-1,3)\}.
  	\end{align*}
  	If $(t_1,t_2,t_3,t_4)\in S_8$, then for $n\geq 0$,
  	\begin{align*}
  		a_{1^{t_1}2^{t_2}3^{t_3}4^{t_4}}(2n)& >0,\quad 
  		a_{1^{t_1}2^{t_2}3^{t_3}4^{t_4}}(6n+1) <0,\\
  		a_{1^{t_1}2^{t_2}3^{t_3}4^{t_4}}(6n+3)& <0,\quad
  			a_{1^{t_1}2^{t_2}3^{t_3}4^{t_4}}(6n+5)=0.
  	\end{align*}
  \end{theorem}
  
  \noindent{\it Proof}.  By \eqref{1-1},
   \begin{align}
   	\sum_{n=0}^\infty a_{1^12^{-4}3^{-3}4^4}(n)
   	 q^n=\frac{f_1f_4^4}{f_2^4f_3^3}. \label{6-1}
   \end{align}
   Xia and Yao \cite[Lemma 2.5]{Xia-Yao}
    proved that 
    \begin{align}
    	\frac{f_1}{f_3^3}=
    	\frac{f_2f_4^2f_{12}^2
    	 }{f_6^7}-q\frac{f_2^3f_{12}^6}{
    	 f_4^2f_6^9}. \label{6-2}
    \end{align}
 Substituting \eqref{6-2} into \eqref{6-1}, we can 
  prove that 
      \begin{align}
   	\sum_{n=0}^\infty a_{1^12^{-4}3^{-3}4^4}(2n)
   	q^n=\frac{f_2^6f_6^2}{f_1^3f_3^7}\succcurlyeq
   	\sum_{n=0}^\infty q^n , \label{6-3}
   \end{align}
   and 
        \begin{align}
   	\sum_{n=0}^\infty a_{1^12^{-4}3^{-3}4^4}(2n+1)
   	q^n=-\frac{f_2^2f_6^6}{f_1f_3^9}. \label{6-4}
   \end{align}
   Substituting \eqref{3-15}
    into \eqref{6-4}, we have 
         \begin{align*}
    	\sum_{n=0}^\infty a_{1^12^{-4}3^{-3}4^4}(2n+1)
    	q^n=-\frac{f_6^7f_9^2}{
    		f_3^{10} f_{18}}- q\frac{f_6^6 f_{18}^2}{f_3^9 f_{9}} ,
    \end{align*}
which implies that 
      \begin{align}
	\sum_{n=0}^\infty a_{1^12^{-4}3^{-3}4^4}(6n+1)
	q^n&=-\frac{f_2^7f_3^2}{
		f_1^{10} f_{6}}\preccurlyeq -\sum_{n=0}^\infty q^n  , \label{6-5}\\
		\sum_{n=0}^\infty a_{1^12^{-4}3^{-3}4^4}(6n+3)
		q^n&=-\frac{f_2^6f_6^2}{f_1^9f_3}  \preccurlyeq -\sum_{n=0}^\infty q^n,\label{6-6}\\
		\sum_{n=0}^\infty a_{1^12^{-4}3^{-3}4^4}(6n+5)
		q^n&=0 .\label{6-7}
\end{align}
    It follows from \eqref{6-3} and 
    \eqref{6-5}--\eqref{6-7} that Theorem \ref{Th-6-1}
     is true when $(t_1,t_2,t_3,t_4)=(1,-4,-3,4)$.
     
     The generating function of $a_{1^32^{-3}3^{-1}4^3}(n)$
      is 
       \begin{align}
     	\sum_{n=0}^\infty a_{1^32^{-3}3^{-1}4^3}(n)
     	q^n=\frac{f_1^3 f_4^3}{f_2^3f_3}. \label{6-8}
     \end{align}
     Using  \eqref{3-15}, \eqref{4-63} and \eqref{6-8}
      and the same method 
       in the proof of   \eqref{6-3}
        and \eqref{6-5}--\eqref{6-7}, we can prove that 
       Theorem  \ref{Th-6-1}
       is true when $(t_1,t_2,t_3,t_4)=(3,-3,-1,3)$.
        This completes the proof of Theorem \ref{Th-6-1}.
         \qed 
       
     \begin{theorem}\label{Th-6-2}
  	For $n\geq 0$, 
  	\[
  	a_{1^{1}2^{-3}3^{-3} 4^{2}}(2n)>0,\
  	a_{1^{1}2^{-3}3^{-3} 4^{2}}(6n+1)<0,\     a_{1^{1}2^{-3}3^{-3} 4^{2}}(6n+3)=0 ,\  a_{1^{1}2^{-3}3^{-3} 4^{2}}(6n+5)=0 .
  	\]
  \end{theorem}
  
  \noindent{\it Proof.}
   The generating function 
    of  $a_{1^{1}2^{-3}3^{-3} 4^{2}}(n)$ is 
        \begin{align}\label{6-9}
  \sum_{n=0}^\infty   
  a_{1^{1}2^{-3}3^{-3} 4^{2}}(n)q^n = \frac{f_1 f_4^2}{ f_2^3f_3^3}.
  \end{align}
   Substituting \eqref{6-2} into \eqref{6-9}, we get
         \[
  \sum_{n=0}^\infty   
  a_{1^{1}2^{-3}3^{-3} 4^{2}}(n)q^n = \frac{f_4^4f_{12}^2
   }{ f_2^2f_6^7}-q\frac{f_{12}^6}{f_6^9},
  \]
  from which, we can prove Theorem \ref{Th-6-2}.
   \qed

     \begin{theorem}\label{Th-6-3}
  	For $n\geq 0$, 
  	\begin{align*}
   a_{1^{1}2^{-2}3^{-3} }(3n)&>0,\
   a_{1^{1}2^{-2}3^{-3} }(6n+1)<0,\    a_{1^{1}2^{-2}3^{-3} }(6n+2)>0,\\
     a_{1^{1}2^{-2}3^{-3} }(6n+4)&=0 ,\   a_{1^{1}2^{-2}3^{-3} }(6n+5)<0 .
  	\end{align*}
  \end{theorem}
  
 \begin{remark}
 	 Theorem \ref{Th-6-3}
 	 was first proved by  
 	 Kane and Sharma \cite{Kane-Sharma}
 	 using the Hardy-Ramanujan-Rademacher circle method.
 \end{remark}  
  
  \noindent{\it Proof.}
   By \eqref{1-1}, 
       \begin{align}\label{6-10}
  \sum_{n=0}^\infty   
  a_{1^{1}2^{-2}3^{-3} }(n)q^n = \frac{f_1 }{ f_2^2f_3^3}.
  \end{align}
  Hirschhorn and Sellers \cite{Hirschhorn-Sellers} proved 
   that 
       \begin{align}\label{6-11}
 \frac{f_1}{f_2^2}=\frac{f_3^2f_9^3}{
  f_6^6}-q\frac{f_3^3f_{18}^3}{f_6^7}+q^2\frac{f_3^4f_{18}^6
   }{f_6^8f_9^3}.
   \end{align}
  Thanks to \eqref{6-10} and \eqref{6-11},
   we can prove that 
    \begin{align} 
 \sum_{n=0}^\infty   
 a_{1^{1}2^{-2}3^{-3} }(3n)q^n &= \frac{f_3^3 }{f_1 f_2^6 }\succcurlyeq \sum_{n=0}^\infty q^n
, \label{6-12}\\
 \sum_{n=0}^\infty   
 a_{1^{1}2^{-2}3^{-3} }(6n+1)q^n &=- \frac{f_3^3 }{ f_1^7 }
  \preccurlyeq -\sum_{n=0}^\infty q^n,  \label{6-13}\\
 \sum_{n=0}^\infty    a_{1^{1}2^{-2}3^{-3} }(6n+4)q^n &=0,  \label{6-13-1} 
\end{align}
and 
  \begin{align} 
 \sum_{n=0}^\infty   
 a_{1^{1}2^{-2}3^{-3} }(3n+2)q^n = \frac{f_1f_6^6 }{f_2^8f_3^3 }.  \label{6-14}
 \end{align}
 Replacing $q$ by $-q$ in
 \eqref{6-14} yields 
  \begin{align} 
 \sum_{n=0}^\infty  (-1)^n  
 a_{1^{1}2^{-2}3^{-3} }(3n+2)q^n = \frac{ f_3^3f_{12}^3 }{f_1f_2^5f_4f_6^3 }\succcurlyeq \sum_{n=0}^\infty q^n. \label{6-15}
 \end{align}
Theorem \ref{Th-6-3} follows from
\eqref{6-12}, \eqref{6-13}, \eqref{6-13-1} and \eqref{6-15}. 
 This completes the proof. \qed

			            	\section{Sign changes with period 8}
			          
	     \begin{theorem}\label{Th-7-1}
		For $n\geq 0$, 
		\begin{align*}
		a_{1^{4}  4^{-5}}(8n)&>0,\
		a_{1^{4}  4^{-5}}(8n+2)>0,\ a_{1^{4}  4^{-5}}(4n+3)>0,\\
		a_{1^{4}  4^{-5}}(4n+1)&<0,\  
		  a_{1^{4}  4^{-5}}(8n+4)=0,\  a_{1^{4}  4^{-5}}(8n+6)=0  .
		\end{align*}
	\end{theorem}		          
			          
\noindent{\it Proof.}
 Thanks to \eqref{1-1}, 
			            \begin{align}\label{7-1}
	\sum_{n=0}^\infty   
	a_{1^{4}  4^{-5}}(n)q^n = \frac{f_1^4 }{  f_4^5 }.
	\end{align}
	Substituting \eqref{4-68} into \eqref{7-1} yields 
\[
	\sum_{n=0}^\infty   
	a_{1^{4}  4^{-5}}(n)q^n = \frac{f_4^5}{f_2^2f_8^4}-4q\frac{f_2^2f_8^4}{f_4^7},
	\]
	from which, we have 
   \begin{align}\label{7-2}
	\sum_{n=0}^\infty   
	a_{1^{4}     4^{ -5}}(2n)q^n = \frac{f_2^5}{f_1^2f_4^4} 
	\end{align}
	and 
	   \begin{align}\label{7-3}
	\sum_{n=0}^\infty   
	a_{1^{4}     4^{ -5}}(2n+1)q^n =  -4 \frac{f_1^2f_4^4}{f_2^7}.
	\end{align}
	
Substituting \eqref{m-1}  into \eqref{7-2}, we can deduce the following generating 
 functions 
 \begin{align}
	\sum_{n=0}^\infty   
a_{1^{4}     4^{ -5}}(8n)q^n &=    \frac{f_2^5}{f_1^4f_4^2}
\succcurlyeq
\sum_{n=0}^\infty q^n, \label{7-4}
\\
	\sum_{n=0}^\infty   
a_{1^{4}     4^{ -5}}(8n+2)q^n &=   2 \frac{f_4^2}{f_1^2f_2}\succcurlyeq
\sum_{n=0}^\infty q^n, \label{7-5}
\\
	\sum_{n=0}^\infty   
a_{1^{4}     4^{ -5}}(8n+4)q^n &=  0, \label{7-6}
\\
\sum_{n=0}^\infty   
a_{1^{4}     4^{ -5}}(8n+6)q^n &=  0. \label{7-8}
 \end{align}
 Substituting \eqref{4-40} into \eqref{7-3}, we obtain the generating
  functions of $a_{1^{4}     4^{ -5}}(4n+1)$ 
   and  $a_{1^{4}     4^{ -5}}(4n+3)$: 
   \begin{align}
   	\sum_{n=0}^\infty   
   	a_{1^{4}     4^{ -5}}(4n+1)q^n &= -4\frac{ f_2^2f_4^5}{f_1^6f_8^2}\preccurlyeq
   	- \sum_{n=0}^\infty q^n, \label{7-9}
   	\\
   	  	\sum_{n=0}^\infty   
   	a_{1^{4}     4^{ -5}}(4n+3)q^n &= 8\frac{ f_2^4f_8^2}{f_1^6f_4}\succcurlyeq  \sum_{n=0}^\infty q^n.
   	\label{7-10}
   \end{align}
	Theorem  \ref{Th-7-1} follows from \eqref{7-4}--\eqref{7-10}.
	 This completes the proof of Theorem \ref{Th-7-1}.
	  \qed  
	
	\begin{theorem}\label{Th-7-2}
		Define 
		\begin{align*}
			S_9:=&\{( 4,0,0,-4),( 4,0,0,-3),( 4,0,0,-2),
			(4,1,0,-6),(4,1,0,-5),(4,1,0,-4), (4,1,0,-3)  \}. 
		\end{align*}
		If $(t_1,t_2,t_3,t_4)\in S_9$, then for $n\geq 0$,
		\begin{align*}
			a_{1^{t_1}2^{t_2}3^{t_3}4^{t_4}}(8n)&> 0, \ 
				a_{1^{t_1}2^{t_2}3^{t_3}4^{t_4}}(8n+2) > 0,\ 
					a_{1^{t_1}2^{t_2}3^{t_3}4^{t_4}}(4n+3) > 0,\\
						a_{1^{t_1}2^{t_2}3^{t_3}4^{t_4}}(4n+1)&< 0, \ 
					a_{1^{t_1}2^{t_2}3^{t_3}4^{t_4}}(8n+4) < 0,\ 
					a_{1^{t_1}2^{t_2}3^{t_3}4^{t_4}}(8n+6) < 0.
		\end{align*}
	 
	\end{theorem}
	
	\noindent{\it Proof.}
	 By \eqref{1-1},   
	   \begin{align}\label{7-11}
	\sum_{n=0}^\infty   
	a_{1^{4}      4^{ t_4}}(n)q^n =  f_1^4   f_4^{ t_4} .
	\end{align}
	Using \eqref{4-68} and \eqref{7-11}, we can prove that 
		\begin{align}
		\sum_{n=0}^\infty   
		a_{1^{4}      4^{ t_4}}(2n)q^n &
		= \frac{ f_2^{10+t_4}  }{  f_1^2 f_4^4 },
		\label{7-12-1}\\
		\sum_{n=0}^\infty   
		a_{1^{4}      4^{ t_4}}(2n+1)q^n &
		= -4\frac{ f_1^{2} f_4^4 }{  f_2^{2-t_4} }.\label{7-13-1}
	\end{align}
	Substituting \eqref{m-1}
	 into \eqref{7-12-1}, we can show that 
	\begin{align}
			\sum_{n=0}^\infty   
		a_{1^{4}      4^{ t_4}}(4n)q^n &
		 = \frac{ f_1^{5+t_4} f_4^5 }{  f_2^4 f_8^2 },
		 \label{7-12}\\
			\sum_{n=0}^\infty   
		a_{1^{4}      4^{ t_4}}(4n+2)q^n &
		= 2\frac{ f_1^{5+t_4} f_8^2 }{  f_2^2 f_4 }.\label{7-13}
	\end{align}
	Replacing $q$ by $-q$
	 in \eqref{7-12} and \eqref{7-13} yields that 
	  for $-4\leq t_4\leq -2$, 
		\begin{align}
		\sum_{n=0}^\infty  (-1)^n   
		a_{1^{4}      4^{ t_4}}(4n)q^n &
		= \frac{ f_2^{11+3t_4}  }{  f_1^{5+t_4}f_4^{t_4} f_8^2 }
		\succcurlyeq \sum_{n=0}^\infty q^n,
	\label{7-15}	\\
		\sum_{n=0}^\infty  (-1)^n  
		a_{1^{4}      4^{ t_4}}(4n+2)q^n &
		= 2\frac{ f_2^{13+3t_4} f_8^2 }{  f_1^{5+t_4} f_4^{6+t_4} }\succcurlyeq
		 \sum_{n=0}^\infty q^n.
		\label{7-16}
	\end{align}
	Based on  \eqref{4-40}
 and  \eqref{7-13-1}, we can prove that 	  for $-4\leq t_4\leq -2$, 
		\begin{align}
		\sum_{n=0}^\infty   
		a_{1^{4}      4^{ t_4}}(4n+1)q^n &
		= -4\frac{ f_2^2f_4^5}{  f_1^{1-t_4} f_8^2 } \preccurlyeq -  \sum_{n=0}^\infty q^n \label{7-17}
	\end{align}
	and 
		\begin{align}
		\sum_{n=0}^\infty   
		a_{1^{4}      4^{ t_4}}(4n+3)q^n &
		= 8 \frac{ f_2^4f_8^2}{  f_1^{1-t_4} f_4 }\succcurlyeq  \sum_{n=0}^\infty q^n.\label{7-18}
	\end{align}

	The generating function of  $a_{1^{4}  2^1    4^{ t_4}}(n)$
	 is 
	 \begin{align}
	 	\sum_{n=0}^\infty a_{1^{4}  2^1    4^{ t_4}}(n)q^n=f_1^4f_2f_4^{t_4}. \label{7-19}
	 \end{align}
Substituting  \eqref{4-68} into  \eqref{7-19}, we can prove that 
	\begin{align}
		\sum_{n=0}^\infty   
		a_{1^{4}  2^1    4^{ t_4}}(2n)q^n &
		=\frac{1}{f_1}\cdot  \frac{ f_2^{10+t_4}  }{   f_4^4 },
		\label{7-20}\\
		\sum_{n=0}^\infty   
		a_{1^{4}  2^1    4^{ t_4}}(2n+1)q^n &
		= -4f_1\cdot f_1^2 \cdot \frac{   f_4^4 }{  f_2^{2-t_4} }.\label{7-21}
	\end{align}
	Substituting \eqref{7-22} into \eqref{7-20}, we get 
		\begin{align}
		\sum_{n=0}^\infty   
		a_{1^{4}  2^1    4^{ t_4}}(4n)q^n &
		= \frac{ f_1^{7+t_4}  }{  f_2^2  }S(-q) \label{7-24}
	\end{align}
	and 
		\begin{align}
		\sum_{n=0}^\infty   
		a_{1^{4}  2^1    4^{ t_4}}(4n+2)q^n &
		= \frac{ f_1^{7+t_4}  }{  f_2^2  }T(-q). \label{7-25}
	\end{align}
	Replacing $q$ by $-q$ in \eqref{7-24} and \eqref{7-25} yields 
	that for $-6\leq t_4\leq -3$,
			\begin{align}
		\sum_{n=0}^\infty (-1)^n   
		a_{1^{4}  2^1    4^{ t_4}}(4n)q^n &
		= \frac{ f_2^{19+3t_4}  }{  f_1^{7+t_4} f_4^{7+t_4}  }S(q) \succcurlyeq
		\sum_{n=0}^\infty q^n \label{7-26}
	\end{align}
	and 
				\begin{align}
		\sum_{n=0}^\infty (-1)^n   
		a_{1^{4}  2^1    4^{ t_4}}(4n+2)q^n &
		= \frac{ f_2^{19+3t_4}  }{  f_1^{7+t_4} f_4^{7+t_4}  }T(q)
		\succcurlyeq  
		\sum_{n=0}^\infty q^n .\label{7-27}
	\end{align}
	Substituting \eqref{4-40}
	 and \eqref{7-23} into \eqref{7-21}, we can prove that 
	  for $-6\leq t_4\leq -3$, 
	 			\begin{align}
	 	\sum_{n=0}^\infty   
	 	a_{1^{4}  2^1    4^{ t_4}}(4n+1)q^n & 
	 	= -4\frac{f_2^3f_4^5}{f_1^{1-t_4}f_8^2} S(-q)
	 	-8q\frac{f_2^5f_8^2}{f_1^{1-t_4}f_4}T(-q)
	  \preccurlyeq - 
	 	\sum_{n=0}^\infty q^n \label{7-28}
	 \end{align}
	 and 
	 		\begin{align}
	 	\sum_{n=0}^\infty    
	 	a_{1^{4}  2^1    4^{ t_4}}(4n+3)q^n & 
	 	= 4\frac{f_2^3f_4^5}{f_1^{1-t_4}f_8^2} T(-q)
	 	+8\frac{f_2^5f_8^2}{f_1^{1-t_4}f_4}S(-q)
	  \succcurlyeq  
	 	\sum_{n=0}^\infty q^n .\label{7-29}
	 \end{align}
	 Theorem \ref{Th-7-2}
	  follows from \eqref{7-15}--\eqref{7-18}
	   and \eqref{7-26}--\eqref{7-29}. \qed

		\begin{theorem}\label{Th-7-3}
		For $n\geq 0$, 
		\begin{align*}
			a_{1^{4} 2^2    4^{ -10}}(8n)& >0,\ 
			 	a_{1^{4} 2^2    4^{ -10}}(4n+3) >0,\ 
			a_{1^{4} 2^2    4^{ -10}}(4n+2) =0, \\
			a_{1^{4} 2^2    4^{ -10}}(8n+4)& =0,\ 
			a_{1^{4} 2^2    4^{ -10}}(4n+1)<0.
		\end{align*}

	\end{theorem}
	
	\noindent{\it Proof.}
	 In view of \eqref{1-1}, 
		     \begin{align}\label{7-30}
	\sum_{n=0}^\infty   
	a_{1^{4} 2^2    4^{ -10}}(n)q^n = \frac{f_1^4 f_2^2 }{  f_4^{10} }. 
	\end{align}
	Substituting \eqref{4-68}
	 into \eqref{7-30}, we can prove that 
			     \begin{align}\label{7-31}
	\sum_{n=0}^\infty   
	a_{1^{4} 2^2    4^{ -10}}(2n)q^n = \frac{1}{  f_4^{4} }
	\end{align}
	and 
			     \begin{align}\label{7-32}
		\sum_{n=0}^\infty   
		a_{1^{4} 2^2    4^{ -10}}(2n+1)q^n = -4
		 \frac{f_1^4f_4^4}{  f_2^{12} }. 
	\end{align}
		Substituting \eqref{4-68}
	into \eqref{7-32}, we can prove that 
	\begin{align}\label{7-33}
	\sum_{n=0}^\infty   
	a_{1^{4} 2^2    4^{ -10}}(4n+1)q^n = -4
	 \frac{f_2^{14}}{  f_1^{14}f_4^{4} } \preccurlyeq -\sum_{n=0}^\infty
	 q^n
	\end{align}
	and 
		\begin{align}\label{7-34}
	\sum_{n=0}^\infty   
	a_{1^{4} 2^2    4^{ -10}}(4n+3)q^n = 16 
	\frac{f_2^{2}f_4^4 }{  f_1^{10}  } \succcurlyeq
	 \sum_{n=0}^\infty
	 q^n.
	\end{align}
	Theorem \ref{Th-7-3}
	 follows from \eqref{7-31}, 
	 \eqref{7-33} and \eqref{7-34}. \qed 
	
		\begin{theorem}\label{Th-7-4-0}
	For $-9\leq t_4 \leq -2$ and $n\geq 0$, 
			\begin{align*}
		a_{1^{4}2^{2}4^{t_4}}(8n)& >0,\quad 	a_{1^{4}2^{2}4^{t_4}}(4n+3) >0,\ 
		a_{1^{4}2^{2}4^{t_4}}(4n+2 )=0 ,\\
		 a_{1^{4}2^{2}4^{t_4}}(8n+4 ) &<0, 
	\quad 	a_{1^{4}2^{2}4^{t_4}}(4n+1 )<0.
	\end{align*} 
	\end{theorem}
	
	\noindent{\it Proof.}
	 The generating function of  $	a_{1^{4}2^{2}4^{t_4}}(n)$
	  is 
	  \begin{align}
	  	\sum_{n=0}^\infty
	  	 	a_{1^{4}2^{2}4^{t_4}}(n)q^n=f_1^4f_2^2f_4^{t_4}.
	  	 	\label{7-35}
	  \end{align}
		Substituting \eqref{4-68}
	into \eqref{7-35}, we can prove that 
	  \begin{align}
	\sum_{n=0}^\infty   
	a_{1^{4}2^{2}4^{t_4}}(4n)q^n &=   
	\frac{f_1^{10+t_4} }{  f_2^{4}  }, \label{7-36}\\
		\sum_{n=0}^\infty   
	a_{1^{4}2^{2}4^{t_4}}(4n+2)q^n &=   
	0, \label{7-37}\\
		\sum_{n=0}^\infty   
	a_{1^{4}2^{2}4^{t_4}}(2n+1)q^n &=   
	-4f_1^4f_2^{t_4-2}f_4^4. \label{7-38}
	\end{align}
	Replacing $q$ by $-q$
	 in \eqref{7-36} yields 
	 \begin{align}
	 		\sum_{n=0}^\infty   (-1)^n 
	 	a_{1^{4}2^{2}4^{t_4}}(4n)q^n &=   
	 	\frac{f_2^{26+3t_4} }{  f_1^{10+t_4} f_4^{10+t_4}  }
	 	 \succcurlyeq \sum_{n=0}^\infty
	 	 q^n.\label{7-39}
	 \end{align}
 		Substituting \eqref{4-68}
 into \eqref{7-38}, we deduce that 
\begin{align}
	\sum_{n=0}^\infty   
	a_{1^{4}2^{2}4^{t_4}}(4n+1)q^n =   
	-4\frac{f_2^{14}}{f_1^{4-t_4}f_4^4} \preccurlyeq -\sum_{n=0}^\infty
	q^n \label{7-40}
	\end{align}
	and 
\begin{align}
	\sum_{n=0}^\infty   
	a_{1^{4}2^{2}4^{t_4}}(4n+3)q^n =   
	16\frac{f_2^{2}f_4^4 }{f_1^{-t_4}} \succcurlyeq \sum_{n=0}^\infty
	q^n. \label{7-41}
	\end{align}
	Theorem \ref{Th-7-4-0} follows from \eqref{7-37} and 
  \eqref{7-39}--\eqref{7-41}. \qed

	\begin{theorem}\label{Th-7-4}
		Define 
		\begin{align*}
			S_{10}:=&\{ (4,3,-4),  (4,4,-4), (4,4,-3)  \}. 
		\end{align*}
		If $(t_1,t_2, t_4)\in S_{10}$, then for $n\geq 0$,
			\begin{align*}
			a_{1^{t_1}2^{t_2} 4^{t_4}}(8n)& >0,\quad 
			 a_{1^{t_1}2^{t_2} 4^{t_4}}(4n+3) >0,\quad  a_{1^{t_1}2^{t_2} 4^{t_4}}(8n+6) >0, \\
			a_{1^{t_1}2^{t_2}4^{t_4}}(4n+1 ) &<0, \quad 
			  a_{1^{t_1}2^{t_2} 4^{t_4}}(8n+2) <0, 
			 \quad  a_{1^{t_1}2^{t_2} 4^{t_4}}(8n+4) <0.  
		\end{align*}
	\end{theorem}
	
\noindent{\it Proof.}
 The generating function
  of $a_{1^42^3 4^{-4}}(n)$ is
  \begin{align*}
  	\sum_{n=0}^\infty
  	 a_{1^42^3 4^{-4}}(n) q^n=f_1^4f_2^3f_4^{-4}.
  \end{align*}
  Substituting \eqref{4-68}
   into the above identity, we can prove that 
  \begin{align}\label{7-42}
	\sum_{n=0}^\infty
	a_{1^42^3 4^{-4}}(2n) q^n=f_1\cdot \frac{f_2^{6}}{f_4^4}
\end{align}	
and 
  \begin{align}\label{7-43}
	\sum_{n=0}^\infty
	a_{1^42^3 4^{-4}}(2n+1) q^n=-4f_1\cdot f_1^4\cdot  \frac{f_4^{4}}{f_2^{6}}. 
\end{align}	
  Substituting \eqref{7-23} 
into \eqref{7-42}, we can prove that 	
  \begin{align}\label{7-44}
	\sum_{n=0}^\infty
	a_{1^42^3 4^{-4}}(4n) q^n=  \frac{f_1^{6}}{f_2^3}S(-q)
\end{align}	
and 
  \begin{align}
	\sum_{n=0}^\infty
	a_{1^42^3 4^{-4}}(4n+2) q^n=-  \frac{f_1^{6}}{f_2^3}T(-q).
	 \label{7-45}
\end{align}	
Replacing $q$ by $-q$ in \eqref{7-44} and \eqref{7-45}, we see that 
  \begin{align} \label{7-46}
	\sum_{n=0}^\infty (-1)^n 
	a_{1^42^3 4^{-4}}(4n) q^n=  \frac{f_2^{15}}{f_1^{6}f_4^{6}}S(q)\succcurlyeq 
	\sum_{n=0}^\infty q^n
\end{align}	
and 
\begin{align}
	\sum_{n=0}^\infty (-1)^n 
	a_{1^42^3 4^{-4}}(4n+2) q^n=- \frac{f_2^{15}}{f_1^{6}f_4^{6}}T(q)  \preccurlyeq -
	 \sum_{n=0}^\infty q^n.
	 \label{7-47}
\end{align}	
Substituting \eqref{7-23} and \eqref{4-68}
   into \eqref{7-43}, we can prove that 
\begin{align} 	 \label{7-48}
	\sum_{n=0}^\infty  
	a_{1^42^3 4^{-4}}(4n+1) q^n=
	-4\frac{f_2^{15}}{f_1^{8}f_4^4}S(-q)-16q\frac{ f_2^3f_4^4}{f_1^{4}}T(-q)\preccurlyeq -\sum_{n=0}^\infty q^n
\end{align}	
and 
\begin{align} 	 \label{7-49}
	\sum_{n=0}^\infty  
	a_{1^42^3 4^{-4}}(4n+3) q^n=
	 4\frac{f_2^{15}}{f_1^{8}f_4^4}T(-q)+16\frac{ f_2^3f_4^4}{f_1^{4}}S(-q) \succcurlyeq \sum_{n=0}^\infty q^n.
\end{align}

Using  \eqref{4-68} and the same method for proving \eqref{7-42} 
 and \eqref{7-43}, we can prove that 
  \begin{align}\label{7-50}
 	\sum_{n=0}^\infty
 	a_{1^42^4 4^{t_4}}(2n) q^n=f_1^2\cdot \frac{f_2^{10+t_4}}{f_4^4}
 \end{align}	
 and 
 \begin{align}\label{7-51}
 	\sum_{n=0}^\infty
 	a_{1^42^4 4^{t_4}}(2n+1) q^n=-4f_1^2\cdot f_1^4\cdot  \frac{f_4^{4}}{f_2^{2-t_4}}.
 \end{align}	
 Substituting \eqref{4-68}
  into \eqref{7-50}, we can obtain 
   \begin{align}\label{7-52}
 	\sum_{n=0}^\infty
 	a_{1^42^4 4^{t_4}}(4n) 
 	q^n= \frac{f_1^{11+t_4} f_4^5 }{f_2^6f_8^2}
 \end{align}	
 and 
   \begin{align}\label{7-53}
 	\sum_{n=0}^\infty
 	a_{1^42^4 4^{t_4}}(4n+2) q^n=-2
 	  \frac{f_1^{11+t_4}f_8^2 }{f_2^4 f_4}.
 \end{align}	
 Replacing $q$ by $-q$
  in \eqref{7-52}
   and \eqref{7-53}, we deduce that 
    for $t_4\in\{-3,-4\}$, 
\begin{align}\label{7-54}
	\sum_{n=0}^\infty (-1)^n 
	a_{1^42^4 4^{t_4}}(4n) 
	q^n= \frac{ f_2^{27+3t_4}}{f_1^{11+t_4}f_4^{6+t_4}f_8^2}\succcurlyeq 
	\sum_{n=0}^\infty q^n 
\end{align}	
and 
\begin{align}\label{7-55}
	\sum_{n=0}^\infty (-1)^n 
	a_{1^42^4 4^{t_4}}(4n+2) q^n=-2
	\frac{ f_2^{29+3t_4} f_8^2
	 }{f_1^{11+t_4} f_4^{12+t_4}}\preccurlyeq 	-\sum_{n=0}^\infty q^n .
\end{align}	
Substituting \eqref{4-40}
and \eqref{4-68}
 into \eqref{7-51}, we can prove  that for $t_4\in\{-3,-4\}$, 
\begin{align}\label{7-56}
	\sum_{n=0}^\infty  
	a_{1^42^4 4^{t_4}}(4n+1) q^n=-4
	\frac{f_2^{12}f_4}{f_1^{3-t_4}f_8^2}-32q\frac{f_2^2f_4^3f_8^2}{
	f_1^{-1-t_4}}\preccurlyeq 	-\sum_{n=0}^\infty q^n  
\end{align}	
and 
\begin{align}\label{7-57}
	\sum_{n=0}^\infty  
	a_{1^42^4 4^{t_4}}(4n+3) q^n=16
	\frac{f_4^9}{f_1^{-1-t_4}f_8^2}+8\frac{f_2^{14}f_8^2}{
		f_1^{3-t_4}f_4^5} \succcurlyeq   \sum_{n=0}^\infty q^n .
\end{align}	
Theorem \ref{Th-7-4}
 follows from \eqref{7-46}--\eqref{7-49}
  and \eqref{7-54}--\eqref{7-57}. \qed

       	\section{Sign changes with period 12}

		\begin{theorem}\label{Th-8-1}
	For   $n\geq 0$, 
	\begin{align*}
		a_{1^{-3} 2^9 3^{1}    4^{ -6}}(12n)& >0,\quad 
			a_{1^{-3} 2^9 3^{1}    4^{ -6}}(4n+1)>0,\quad  a_{1^{-3} 2^9 3^{1}    4^{ -6}}(4n+3)<0,\\
	a_{1^{-3} 2^9 3^{1}    4^{ -6}}(6n+2) &=a_{1^{-3} 2^9 3^{1}    4^{ -6}}(6n+4)=0, \quad 
		a_{1^{-3} 2^9 3^{1}    4^{ -6}}(12n +6)<0.
	\end{align*}
\end{theorem}	

\noindent{\it Proof.} The generating function 
 of  $a_{1^{-3} 2^9 3^{1}    4^{ -6}}(n)$ is
 \begin{align}\label{8-1}
	\sum_{n=0}^\infty   
	a_{1^{-3} 2^9 3^{1}    4^{ -6}}(n)q^n = \frac{f_2^9 f_3}{
	 f_1^3f_4^6  }.
	\end{align}
Xia and Yao \cite[Theorem 3.5]{Xia-Yao-2013} proved that 
\begin{align}\label{8-2-1}
\frac{f_3}{f_1^3}
=\frac{f_4^6f_6^3}{f_2^9f_{12}^2}+3q\frac{ f_4^2f_6f_{12}^2
 }{f_2^7}.
\end{align}
Substituting \eqref{8-2-1} into \eqref{8-1}, we get 
 \begin{align} 
	\sum_{n=0}^\infty   
	a_{1^{-3} 2^9 3^{1}    4^{ -6}}(6n)q^n &= \frac{f_1^3}{
		f_2^2 }, \label{8-3}\\
		\sum_{n=0}^\infty   
		a_{1^{-3} 2^9 3^{1}    4^{ -6}}(6n+2)q^n &=0, \label{8-4}\\
			\sum_{n=0}^\infty   
			a_{1^{-3} 2^9 3^{1}    4^{ -6}}(6n+4)q^n &= 0 \label{8-5}
\end{align}
and 
 \begin{align}\label{8-6}
	\sum_{n=0}^\infty   
	a_{1^{-3} 2^9 3^{1}    4^{ -6}}(2n+1)q^n = 3\frac{f_1^2f_3f_6^2}{
		f_2^4  }.
\end{align}
Replacing $q$ by $-q$ in \eqref{8-3} and \eqref{8-6}, we arrive at 
	 \begin{align}\label{8-7}
	\sum_{n=0}^\infty    (-1)^n 
	a_{1^{-3} 2^9 3^{1}    4^{ -6}}(6n)q^n = \frac{ f_2^5}{
		f_1^2f_4^2  } \cdot \frac{f_2^2}{f_1}\cdot \frac{1}{f_4}\succcurlyeq
		\sum_{n=0}^\infty q^n
	\end{align}
	and 
		 \begin{align}\label{8-8}
	\sum_{n=0}^\infty   
(-1)^n 	a_{1^{-3} 2^9 3^{1}    4^{ -6}}(2n+1)q^n = 3\frac{
	f_2^2 f_6^5 }{f_1^2f_3f_4^2f_{12}}\succcurlyeq \sum_{n=0}^\infty q^n.
	\end{align}
	Theorem \ref{Th-8-1} follows from \eqref{8-4}, \eqref{8-5},
	 \eqref{8-7} and \eqref{8-8}. \qed

	 \section*{Statements and Declarations}
	 
	 \noindent{\bf Acknowledgments}. 
	This work was supported by the Qinglan Project of Jiangsu Province. 
	 
	 \noindent{\bf Competing Interests.}
	 The authors declare that they have
	 no conflict of interest.

	 \noindent{\bf Data Availability.} Data sharing is not applicable to this
	 article as no datasets were generated or analyzed during the current
	 study.

\end{document}